\documentclass[a4paper, 10pt]{amsart}
\usepackage{amscd}
\usepackage{mathrsfs}
\usepackage{color}
\usepackage{amssymb,amsmath,amsthm,amsfonts,latexsym}
\usepackage[utf8x]{inputenc}
\usepackage{bbm} 
\usepackage[margin=1in]{geometry}
\usepackage{graphicx}

\usepackage{amsmath,amsfonts,amssymb,amsthm} 
\usepackage{tikz}
\usepackage{hyperref}

\newtheorem{theorem}{Theorem}[section]

\newtheorem{corollary}[theorem]{Corollary}

\newtheorem{claim}[theorem]{Claim}

\newtheorem{proposition}[theorem]{Proposition}

\newtheorem{definition}[theorem]{Definition}
\newtheorem{question}[theorem]{Question}

\newcommand{\comm}[1]{}

\renewcommand{\deg}{\operatorname{deg}}

\newcommand{\bir}{\operatorname{\bf BReg}}
\newcommand{\reg}{\operatorname{\bf Reg}}

\def\leukfrac#1/#2{\leavevmode
               \kern.1em
                \raise.9ex\hbox{\the\scriptfont0 ${}_#1$}
                \hskip -1pt\kern-.1em
                /\kern-.15em\lower.10ex\hbox{\the\scriptfont0 ${}_#2$}}

\makeatletter
\def\diam{\mathop{\operator@font diam}\nolimits}
\makeatother

\begin{document}
\title{Fractional Isomorphism of Graphons}
\author{Jan Greb\' ik and Israel Rocha}
\address{\emph{Grebík}:
Mathematics Institute.
University of Warwick,
Coventry CV4~7AL, UK} 
\email{jan.grebik@warwick.ac.uk}
\address{\emph{Rocha}: 
Institute of Computer Science of the Czech Academy of Sciences.
Pod Vodárenskou v\v{e}\v{z}\'i 2, 182~07, Prague, Czechia.
With institutional support RVO:67985807. }
\email{israelrocha@gmail.com}
\thanks{\emph{Greb\' ik} was supported by the Czech Science Foundation,
grant number GJ16-07822Y, by the grant GAUK 900119
of Charles University and by Leverhulme Research Project Grant
RPG-2018-424.
Part of the work was done while \emph{Greb\' ik} was affiliated with Institute of Computer
Science of the Czech Academy of Sciences, with institutional support RVO:67985807.}
\thanks{\emph{Rocha} was supported by the Czech Science Foundation,
grant number GJ16-07822Y and GA19-08740S}

\begin{abstract}
We work out the theory of fractional isomorphism of graphons as a generalization to the classical theory of fractional isomorphism of finite graphs.
The generalization is given in terms of homomorphism densities of finite trees and it is characterized in terms of distributions on iterated degree measures, Markov operators, weak isomorphism of a conditional expectation with respect to invariant sub-$\sigma$-algebras and isomorphism of certain quotients of given graphons.

\end{abstract}

\maketitle

\section{Introduction}\label{sec:Intro}

\emph{Fractional isomorphism of finite graphs} is an important and well-studied notion in graph theory and combinatorial optimization.
Its importance comes from the fact that it is a relaxation of the notoriously difficult \emph{graph isomorphism problem}, it can be solved in polynomial time and, by a result of Babai, Erd\H{o}s, and Selkov \cite{BabaiErdosSelkow}, it distinguishes almost all non-isomorphic graphs.
In contrast, isomorphism problem is not known to be solvable in polynomial time nor to be NP-complete.\footnote{Recently Babai \cite{Babai} found an algorithm for the isomorphism problem that runs in quasipolynomial time.}
There are plenty of characterizations of fractional isomorphism that use different, seemingly unrelated, properties of graphs.
We summarize some of these characterizations that are relevant for our purposes later in the introduction.
We refer the reader to the book of Scheinerman and Ullman~\cite{FracGraphTheory} for a detailed study of the subject.

In this paper, we define and investigate the graphon counterpart of fractional isomorphism, i.e., \emph{fractional isomorphism of graphons}, and prove several equivalent characterizations.
Graphons, introduced by Borgs, Chayes, Lovász, Sós, Szegedy, and Vesztergombi \cite{Lovasz2006,Borgs2008c,Borgs2012c}, emerged as limit objects in the theory of dense graph limits.
The theory of graphons is mostly linked with problems in extremal graph theory and random graphs.
However, it has been successfully applied to solve problems in various areas of combinatorics.
We refer the reader to the beautiful  book of Lov\' asz \cite{LargeNetworks} for more details and examples.

The main contribution of this paper is twofold.
First, we provide a graphon versions of the most important notions that are used as characterizations of fractional isomorphism of finite graphs and show that they are all equivalent for graphons.
Finding graphon counterparts of notions or statements from graph theory is interesting in its own right, e.g. see~\cite{Hladky1,Hladky2}.
Usually it is easy to define the corresponding notion and difficult to provide statements but in our case both tasks turned out to be difficult.
Second, as one of the possible definitions/characterizations of fractional isomorphism of graphons is given via restricting the density vector to finite trees, i.e., graphons $W$ and $U$ are fractionally isomorphic if and only if $t(T,W)=t(T,U)$ for every finite tree $T$, we find this property worth to investigate solely from the graphon point of view.
We describe what similarity must necessarily occur between graphons that have the same tree densities and provide invariants in terms of special measures, called DIDM, that could be computed in cut-distance continuous way.

\subsection{Finite Graphs}

The easiest way to define fractional isomorphism of finite graphs is as a relaxation of the isomorphism problem via doubly stochastic matrices.
For a given graph $G$ denote as $A_G$ the incidence matrix of $G$.
Note that graphs $G$ and $H$ are isomorphic if and only if there is a permutation matrix $P$ such that $A_GP=PA_H$.
We say that a matrix $S$ is a \emph{doubly stochastic} matrix if $S$ has positive entries, i.e., $S\ge 0$, and $S{\bf 1}=S^T{\bf 1}={\bf 1}$.
It is easy to see that every permutation matrix is doubly stochastic.
We say that graphs $G$ and $H$ are \emph{fractionally isomorphic} if there is a doubly stochastic matrix $S$ such that $A_GS=SA_H$.

Next we recall the equivalent concepts that we use in this paper.
We start with \emph{iterated degree sequences}.
For a graph $G$ we denote as $\operatorname{N}(v)$ the set of all neighbors of a vertex $v\in V(G)$ in $G$ and put $\deg_G(v)=|\operatorname{N}(v)|$.
Define, as multisets,
\begin{equation}\label{eq:degreeseq}
\operatorname{D}_1(G)=\{\deg_G(v):v\in V(G)\} \ \text{ and } \ \operatorname{d}_1(v)=\{\deg_G(w):w\in N(v)\}
\end{equation}
and then inductively for every $k\in \mathbb{N}$
\begin{equation}\label{eq:degreeseq2}
\operatorname{D}_{k+1}(G)=\{\operatorname{d}_k(v):v\in V(G)\} \ \text{ and } \ \operatorname{d}_{k+1}(v)=\{\operatorname{d}_k(w):w\in N(v)\}.
\end{equation}
Finally, we define the iterated degree sequence of a graph $G$ as $\operatorname{D}(G)=(\operatorname{D}_k(G))_{k\in \mathbb{N}}$.
It is a result of Tinhofer \cite{Tinhofer86,Tinhofer1991} that $G$ and $H$ are fractionally isomorphic if and only if $D(G)=D(H)$.

An {\it equitable partition}\footnote{Here and throughout the paper we refer to the definition of equitable partition from~\cite{FracGraphTheory}, not to be confused with the definition of an equitable partition in the formulation of Szemerédi's regularity lemma.} of a graph $G$ is a sequence $\mathcal{C}=\{C_j\}_{j\in [k]}$ that is a non-trivial partition of $V(G)$, i.e., $C_j\not=\emptyset$ for every $j\in [k]$, $\bigsqcup_{j\in [k]} C_j=V(G)$, and $\deg_{G}(v_0,C_j)=\deg_{G}(v_1,C_j)$ for every $i,j\in [k]$ such that $v_0,v_1\in C_i$.
It means that each induced subgraph $G[C_i]$ must be regular and each of the bipartite graphs $G[C_i,C_j]$ must be biregular.
The {\it parameters of $\mathcal{C}$} are given by a pair $({\bf n},C)$, where ${\bf n}$ is an $k$-dimensional vector and $C$ is an $k\times k$ square matrix such that ${\bf n}(j)=|C_j|$ and $C(i,j)=\deg_{G}(v,C_j)$, for some $v\in C_i$, i.e, the parameters of $\mathcal{C}$ are the numerical information that we can read from $\mathcal{C}$.
If $G$ and $H$ admit equitable partitions $\mathcal{C}$ and $\mathcal{D}$ that can be indexed in such a way that the parameters of $\mathcal{C}$ and $\mathcal{D}$ are the same, then we say that $G$ and $H$ have a {\it common equitable partition}.
It is a result of Ramana, Scheinerman and Ullman~\cite{FracIsoScheinerman} that $G$ and $H$ are fractionally isomorphic if and only if they have a common equitable partition.
Prior to this it was shown by Tinhofer \cite{Tinhofer86} that $G$ and $H$ are fractionally isomorphic if and only if they have the same coarsest equitable partition.
Recall that a partition $\mathcal{C}$ is {\it coarser} than a partition $\mathcal{D}$ if every element of $\mathcal{D}$ is a subset of some element of $\mathcal{C}$.
It is not hard to verify that every finite graph admits the {\it coarsest equitable partition}, i.e., equitable partition that is coarser than any other equitable partition.

The last equivalence that we mention is the most surprising one.
For finite graphs $F$ and $G$ we denote as $\operatorname{Hom}(F,G)$ the collection of all homomorphisms from $F$ to $G$.
It is a result of Dell, Grohe and Rattan~\cite{Dell2018LovszMW} that $G$ and $H$ are fractionally isomorphic if and only if $|\operatorname{Hom}(T,G)|=|\operatorname{Hom}(T,H)|$ for every finite tree $T$, see also Dvo\v r\' ak~\cite{Dvorak}.

\subsection{Graphons}

A \emph{graphon} is a symmetric measurable function $W:X\times X\to [0,1]$, where $(X,\mathcal{B})$ is a standard Borel space endowed with a Borel probability measure $\mu$.\footnote{The reason why we use standard Borel spaces and not standard probability spaces (or simply unit interval with the Lebesgue measure as it is usual) is that we work with the space of all Borel measures which is a standard Borel space under the assumption that the base space is standard Borel space.
Also we note that every standard probability space is given as the measure completion of some standard Borel space with a Borel probability measure.}
We write $\mathcal{W}_0$ for the space of all graphons after identifying graphons that are equal almost everywhere.
This makes $\mathcal{W}_0$ a subset of $L^\infty(X\times X,\mu\times \mu)$ and of $L^2(X\times X,\mu\times \mu)$ and one may consider the distances on $\mathcal{W}_0$ induced from the corresponding norms.
However, the most relevant notion of distance for studying graphons as dense graph limits comes from the {\it cut-norm} and is defined as
$$d_\Box(W,U)=\sup_{A,B\subseteq X}\left|\int_{A\times B} (W-U) \ d(\mu\times \mu)\right|,$$
where the supremum runs over all measurable subsets $A,B$ of $X$.
The {\it cut-distance $\delta_\Box$} is then defined as
$$\delta_\Box(W,U)=\inf_{\varphi} d_\Box(W^{\varphi},U),$$
where $W^{\varphi}(x,y)=W(\varphi(x),\varphi(y))$ and the infimum runs over all $\varphi:X\to X$ measure preserving bijections of $X$.
Considering $W^\varphi$ and $W$ to be the same is the measurable analogue of considering two finite graphs the same if they are isomorphic.
However, in the qualitative version given by $d_\Box$ we might get $\delta_\Box(W,U)=0$ while there is no single $\varphi$ such that $W^\varphi=U$.
Therefore, we say that $W$ and $U$ are {\it isomorphic} if we have $\varphi$ such that  $W^\varphi=U$ for some measure preserving bijection $\varphi:X\to X$ and we say that $W$ and $U$ are {\it weakly isomorphic} if $\delta_\Box(W,U)=0$.
Notice that $\delta_\Box$ is only a pseudometric on $\mathcal{W}_0$.
We write $\widetilde{\mathcal{W}}_0$ for the quotient space $\mathcal{W}_0$ modulo weak isomorphism equivalence.
It is easy to see that $\delta_\Box$ is a metric on $\widetilde{\mathcal{W}}_0$ and it is a fundamental result in the theory of graphons that $\left(\widetilde{\mathcal{W}}_0,\delta_\Box\right)$ is a compact metric space, see \cite{Lovasz2006}.

An equivalent description of convergence in the space $\widetilde{\mathcal{W}_0}$ can be obtained via homomorphism densities.
Let $F$ and $G$ be finite graphs.
The {\it homomorphism density of $F$ in $G$} is defined as
$$t(F,G)=\frac{|\operatorname{Hom}(F,G)|}{|V(G)|^{|V(F)|}}.$$
That is, $t(F,G)$ is the probability that a random map of the vertices of $F$ to the vertices of $G$ is a homomorphism.
Note that the notion is invariant under isomorphisms.
The analogous notion for graphons is defined as
$$t(F,W)=\int_{X^{V(F)}} \prod_{\{v,w\}\in E(F)} W(y(v),y(w)) \ d\mu^{\oplus |V(F)|}(y)$$
and it is not hard to see that $t(F,W)=t(F,U)$ whenever $W$ and $U$ are weakly isomorphic.
Remarkably, the authors of \cite{Lovasz2006,Borgs2008c} proved an equivalence between the two types of convergence: a sequence of graphons $W_{n}$ converges to $W$ in the cut-distance topology if and only if for every finite graph $F$ we have $t(F,W_{n})\rightarrow t(F,W)$.

An important way to view graphons is as self-adjoint Hilbert-Schmidt operators on $L^2(X,\mu)$.
Namely, for a graphon $W\in \mathcal{W}_0$ the operator $T_W:L^2(X,\mu)\to L^2(X,\mu)$ is defined as
$$T_W(f)(x)=\int_X W(x,y)f(y) \ d\mu(y),$$
where $f\in L^2(X,\mu)$ and $x\in X$, see \cite[Section~7.5]{LargeNetworks}.

\subsection{Fractional Isomorphism of Graphons}

We use a graphon analogue of the characterization of Dell, Grohe and Rattan mentioned above to define fractional isomorphism of graphons.
This shift from the number of homomorphisms to the homomorhism densities (of trees) when transitioning from graphs to graphons parallels the more classical situation of isomorphisms.
Indeed, we already saw that weak isomorphism of graphons is characterized by homomorphism densities, the finite counterpart to this is a result of Lov\' asz \cite{Lovasz1967} which says that graphs $G$ and $H$ are isomorphic if and only if $|\operatorname{Hom}(F,G)|=|\operatorname{Hom}(F,H)|$ for every finite graph $F$.

\begin{definition}[Fractional Isomorphism of Graphons]
We say that graphons $W$ and $U$ are \emph{fractionally isomorphic} if
$$t(T,U)=t(T,W)$$
for every finite tree $T$.
\end{definition}

It follows from~\cite{Dell2018LovszMW} that this definition extends the definition for finite graphs in the sense that $G$ and $H$ are fractionally isomorphic (as finite graphs) if and only if they have the same number of vertices and $W_G$ and $W_H$,  their graphon representations, are fractionally isomorphic (as graphons).
This is in analogy with the fact that $G$ and $H$ are isomorphic if and only if they have the same number of vertices and $W_G$ and $W_H$ are weakly isomorphic.
Also it is a trivial consequence of the definition that fractional isomorphism is an equivalence relation on $\widetilde{\mathcal{W}}_0$ that is closed in the cut-distance topology.\footnote{If $W_n\xrightarrow{\delta_\Box} W$, $U_n\xrightarrow{\delta_\Box} U$ and $W_n, U_n$ are fractionally isomorphic for every $n\in \mathbb{N}$, then $W$ and $U$ are fractionally isomorphic.}

To state our main result, Theorem~\ref{th:main}, we need to introduce and recall some notions.
We try to keep things informal and rather intuitive in this section.
We start with analogue of doubly stochastic matrices.
An operator $S:L^2(X,\mu)\to L^2(X,\mu)$ is a {\it Markov operator}\footnote{Our main reference for the theory of Markov operators is \cite{eisner2015operator}.
We note that in~\cite{eisner2015operator} Markov operators are defined on $L^1$-spaces rather than on $L^2$-spaces.
The fact that theses notions are the same is explained in Appendix~\ref{App D}.} if $S\ge 0$, i.e., $S(f)\ge 0$ whenever $f\ge 0$, and $S({\bf 1}_X)=S^*({\bf 1}_X)={\bf 1}_X$, where $S^*$ is the adjoint of $S$.

We remind the reader that $(X,\mathcal{B})$ is a standard Borel space and $\mu$ is a Borel probability measure.
A sub-$\sigma$-algebra $\mathcal{C}$ of $\mathcal{B}$ is \emph{$W$-invariant}, where $W$ is a graphon, if $T_W(f)$ is $\mathcal{C}$-measurable whenever $f\in L^2(X,\mu)$ is $\mathcal{C}$-measurable.\footnote{To make this definition formally precise we require $\mathcal{C}$ to be \emph{relatively complete}, i.e., $A\in \mathcal{C}$ whenever there is $A'\in \mathcal{C}$ such that $A\subseteq A'$ and $\mu(A')=0$, see Section~\ref{sec:Basic}.}
We illustrate this notion with a few examples.
If $W$ is \emph{$q$-regular}, i.e., $q=\deg_W(x)=\int_X W(x,{-}) \ d\mu$ for ($\mu$-almost) every $x\in X$, then $\mathcal{C}=\langle\{\emptyset,X\}\rangle$ is $W$-invariant.
If $W$ satisfies $\deg_W(x)\not=\deg_W(y)$ for every $x\not=y\in X$, then the only $W$-invariant sub-$\sigma$-algebra is $\mathcal{B}$.
Another example is connected with the concept of twin-free graphons, see~\cite[Section~13.1.1]{LargeNetworks}.
Define
$$\mathcal{C}_{\operatorname{twin}}=\{B\in \mathcal{B}:x\in B \ \& \ W(x,{-})=W(y,{-}) \ \Rightarrow \ y\in B\}.$$
Then $\mathcal{C}_{\operatorname{twin}}$ is always $W$-invariant and $\mathcal{C}_{\operatorname{twin}}\not=\mathcal{B}$ if and only if $W$ is not twin-free graphon.
We show that for every graphon $W$ there exists the unique minimum $W$-invariant sub-$\sigma$-algebra and we denote it as $\mathcal{C}(W)$.
It is not obvious at this point but $W$-invariant algebras correspond to equitable partitions and $\mathcal{C}(W)$ corresponds to the coarsest equitable partition.

Unlike finite graphs, graphon space is rich enough to allow for averaging and quotients.
Given a sub-$\sigma$-algebra $\mathcal{C}$  of $\mathcal{B}$ we define $W_\mathcal{C}$ as a conditional expectation of $W$ given $\mathcal{C}\times \mathcal{C}$, i.e., $W_\mathcal{C}=\mathbb{E}\left(W|\mathcal{C}\times \mathcal{C}\right).$
In the context of standard Borel spaces it is possible to define a quotient graphon $W/\mathcal{C}$ on a quotient space $(X/\mathcal{C},\mathcal{C}')$ with Borel probability measure $\mu/\mathcal{C}$ that is weakly isomorphic to $W_\mathcal{C}$.
Note that the quotient graphon $W/\mathcal{C}_{\operatorname{twin}}$ is a twin-free version of $W$.

The last concept is inspired by iterated degree sequences.
We describe the first two steps of the analogous iterative construction.
Given a graphon $W$ and $x\in X$ consider the Borel assignment
$$x\mapsto i_{W,1}(x)=\deg_{W}(x)=\int_{X} W(x,y) \ d\mu(y)\in [0,1].$$
This is just the degree map that corresponds to $\deg_G$ in (\ref{eq:degreeseq}).
Note that we can view $i_{W,1}(x)$ as a measure on a one-point space $\{\star\}$ and that the space of all measures on $\{\star\}$ of total mass at most $1$ is naturally isomorphic to $[0,1]$.
Taking the Borel probability measure on $[0,1]$ that is the distribution of degrees of $W$, i.e., the push-forward of $\mu$ via $i_{W,1}$, is the analogue of $D_1$ in (\ref{eq:degreeseq}).
The second step is to assign to a vertex $x$ a measure that is a weighted modification of the distribution of the degrees of $W$ with weights given by $W(x,{-})$.
More precisely we assign to a vertex $x\in X$ a Borel measure $i_{W,2}(x)$ on $[0,1]$ that is defined as
$$i_{W,2}(x)(A)=\int_{i_{W,1}^{-1}(A)}W(x,y) \ d\mu(y).$$
This corresponds to $d_1$ in (\ref{eq:degreeseq2}) and similarly we define the analogue of $D_2$ in (\ref{eq:degreeseq2}) as the push-forward of $\mu$ via $i_{W,2}$, this is a Borel probability measure on the space of all Borel measures on $[0,1]$.

This construction can be iterated to define a Borel map $i_W:X\to \mathbb{M}$, where $i_W(x)$ is an infinite sequence of Borel measures and $\mathbb{M}$ is a compact metric space that is defined independently of $W$ and whose elements we call \emph{iterated degree measures}.
The analogue of an iterated degree sequence is then a distribution $\nu_W$ on $\mathbb{M}$ that is the push-forward of $\mu$ via $i_W$.
We call such distributions \emph{DIDM, distributions on iterated degree measures}, a precise definition is given in Section~\ref{sec:DIDD}.
We show that the assignment $W\to \nu_W$ is continuous when $\mathcal{W}_0$ is endowed with the cut-distance topology and the space of Borel probability measures on $\mathbb{M}$ with the weak* topology.

Now we are ready to state our main result.

\begin{theorem}[Characterizations of Fractional Isomorphism of Graphons]\label{th:main}
Let $W$ and $U$ be graphons.
Then the following are equivalent:
\begin{enumerate}
	\item $t(T,W)=t(T,U)$ for every finite tree $T$,
	\item $\nu_W=\nu_U$,
	\item $W/\mathcal{C}(W)$ and $U/\mathcal{C}(U)$ are isomorphic,
	\item there is a Markov operator $S:L^2(X,\mu)\to L^2(X,\mu)$ such that $T_W\circ S=S\circ T_U$,
	\item there is a $W$-invariant sub-$\sigma$-algebra $\mathcal{C}$ and a $U$-invariant sub-$\sigma$-algebra $\mathcal{D}$ such that $W_\mathcal{C}$ and $U_{\mathcal{D}}$ are weakly isomorphic.
\end{enumerate}
\end{theorem}

Here is a good place to mention that the authors announced in \cite{GrebikRocha}, in a slightly different language, the equivalence of (3)--(5).
Indeed, it was our original motivation to find a graphon analogue of equitable partitions and doubly stochastic matrices.
However, after extending the characterization to (1), that was inspired by \cite{Dell2018LovszMW}, and following suggestions of one of the referees we decided to emphasize the equivalence of (1) and (2) as the main result.

The paper is structured as follows.
In Section~\ref{sec:EssStruc} we describe the essential structure of fractionally isomorphic graphons in the more intuitive language of measurable partitions and in Section~\ref{sec:problems} we collect a few remarks and problems.
The rest of the paper is devoted to the proof of Theorem~\ref{th:main}.
In Section~\ref{sec:strategy} we sketch a strategy of the proof.
In Section~\ref{sec:Basic} we prove basic facts about sub-$\sigma$-algebras, invariant subspaces and the minimum algebra $\mathcal{C}(W)$.
In Section~\ref{sec:DIDD} we construct the space $\mathbb{M}$, define DIDM, and show the correspondence between integral kernels and DIDM.
In Section~\ref{sec:Tree functions} we prove the main technical result about the collection of tree functions $\mathcal{T}$ defined on $\mathbb{M}$.
Finally, in Section~\ref{sec:Proof} we prove Theorem~\ref{th:main}.
In Appendices~\ref{App A},~\ref{App B},~\ref{App C},~\ref{App D}, and ~\ref{App E} we collect several well-known facts about standard Borel spaces, spaces of probability measures, and the connection between sub-$\sigma$-algebras, conditional expectations, and Markov operators that we need in our proof.

We denote as $[n]$ the set $\{1,\dots,n\}$.
We write $\mu^{\oplus k}$ for the product measure of $k$-many copies $\mu$.
All the $L^p$ spaces that we consider in this paper are real and so are the spaces of continuous functions on compact spaces.
\footnote{Even though most of the classical results that we use are traditionally stated for complex $L^p$ spaces, they do hold for real spaces as well.
This is because we work either with real valued integral kernels or Markov operators.}

\section{Structure of Fractionally Isomorphic Graphons}\label{sec:EssStruc}

In this section we describe informally a general construction of a \emph{$V$-biregular blowup} of a graphon $V$ and show that every graphon obtained in this way is fractionally isomorphic to $V$.
The easiest way to describe this construction is in the language of measurable partitions that we used in \cite{GrebikRocha}.
This gives plenty of examples of fractionally isomorphic graphons that are not derived from finite graphs.
On the other hand, Theorem~\ref{th:main} implies that this describes all the examples.
Namely, for every pair of fractionally isomorphic graphons $W$ and $U$ there is a graphon $V$ such that $W$ and $U$ are $V$-biregular blowups.
This uses characterization (5) in Theorem~\ref{th:main}.
In the construction, we use some standard measure theoretic techniques on product spaces.
The reader familiar with these techniques can safely skip, after checking the notation in the next paragraph, to Section~\ref{subsec:uncount}.

An intuitive explanation of the construction is as follows.
Pick a graphon $V$ on a standard Borel space $Y$ with a probability measure $\rho$ and form a space $X$ by blowing up each $y\in Y$ to a copy of the unit interval.
There is a canonical measure on $X$, namely the product measure $\rho\times \lambda$, where $\lambda$ is the Lebesgue measure.
For each $y,z\in Y$ pick a biregular function $\Omega_{y,z}\in \bir_{V(y,z)}$ on $[0,1]$ (see below) and glue them together to create a function $W:X\times X\to [0,1]$.
If the choices are symmetric and measurable in $(y,z)$, then $W$ is a graphon on $X$.  
Any such $W$ is called a \emph{$V$-biregular blowup}.

Before we formalize the definition, we recall the basic concepts.
A partition $\eta$ of a standard Borel space $X$ is \emph{measurable} if there is a Borel map $q:X\to Y$, where $Y$ is a standard Borel space such that $\eta=\{q^{-1}(y)\}_{y\in Y}$.
A typical example of a measurable partition is a partition induced by a projection in a product space, i.e., $X=Y\times [0,1]$ and $\eta=\{\{y\}\times [0,1]\}_{y\in Y}$.
There is a correspondence between measurable partitions and sub-$\sigma$-algebras.

Let $q\in [0,1]$ and define $\bir_q$ to be the space of all measurable functions $U:[0,1]^2\to [0,1]$ such that 
$$q=\int_{[0,1]} U(x,{-}) \ d\lambda=\int_{[0,1]} U({-},x) \ d\lambda$$
for ($\lambda$-almost) every $x\in [0,1]$ and put $\bir=\bigcup_{q\in [0,1]} \bir_q$.
Moreover, let $\reg_q$ be a subset of $\bir_q$ that consists of symmetric functions.

\subsection{Countable case}

Before we present the general construction we start with a graphon $V$ on a countable measure space $(Y,\mathcal{D})$ with a Borel probability measure $\rho$, i.e., $|Y|\le \aleph_0$, $\mathcal{D}$ consists of all subsets of $Y$ and $\rho$ is, after a slight abuse of notation, fully determined by a function $\rho:Y\to [0,1]$ such that $\sum_{y\in Y}\rho(y)=1$.
This corresponds to atomic sub-$\sigma$-algebras and countable measurable partitions.
Here the measurable analogue of equitable partition is easy to digest and so is its connection to invariant sub-$\sigma$-algebras.


Let us start with a trivial case when $|Y|=1$ and $V$ is a constant graphon that attains a value $q\in [0,1]$.
In this case a \emph{$V$-biregular blowup} is any element of $\reg_q$.
It is easy to see that if $W\in \reg_q$, then $\mathcal{C}=\left\langle\{\emptyset, [0,1]\}\right\rangle$ is a $W$-invariant subalgebra and $W/\mathcal{C}=V$.
Therefore elements of $\reg_q$ are pairwise fractionally isomorphic.
In the language of measurable partitions we might say that given $W\in \reg_q$ we consider the trivial partition $\eta=\{[0,1]\}$ of $[0,1]$.
Then $\eta$ satisfies a measurable analogue of the condition from the definition of equitable partition from previous section.
Namely, we have
$$\deg_W(x,[0,1])=\deg_W(x)=q$$
for $\lambda$-almost every $x\in [0,1]$.

Suppose that $Y=\mathbb{N}$ and pick a graphon $V$ on $Y$.
Put $I_i=[0,1]$ and $\lambda_i$ for the Lebesgue measure on $I_i$, where $i\in \mathbb{N}$.
Consider a measure space $X=\bigsqcup_{i\in \mathbb{N}} I_i$ with a Borel probability measure $\mu=\sum_{i\in \mathbb{N}} \rho(i)\lambda_i$.
Note that 
$$\mu(A)=\int_{Y} \lambda_i(A) \ d\rho(i)$$
holds for every Borel set $A\subseteq X$.
Let 
$$\Omega:\mathbb{N}\times \mathbb{N}\to \bir$$
be a map that satisfies $\Omega(i,j)=\Omega(j,i)$ and $\Omega(i,j)\in \bir_{V(i,j)}$.
Now for every such $\Omega$ we define a \emph{$V$-biregular blowup} to be a graphon $W_\Omega$ on $X$ defined as 
$$W_\Omega((i,r),(j,s))=\Omega(i,j)(r,s).$$
Let $\mathcal{C}$ be a sub-$\sigma$-algebra generated by the partition $\eta=\{I_i\}_{i\in \mathbb{N}}$.
It is straightforward to check that $\mathcal{C}$ is $W$-invariant and $W_\Omega/\mathcal{C}=V$.
Therefore any two $V$-biregular blowups are fractionally isomorphic.
It follows from the definition that $\eta$ satisfies  
$$\deg_W((i,r),I_j)=\int_{I_j} W((i,r),(j,s)) \ d \lambda_i(s)=V(i,j)$$
for every $i,j\in \mathbb{N}$ and $\mu$-almost every $r\in I_i$.
This is the measurable analogue of the equitable condition for a countable $Y$.

\subsection{Uncountable case}\label{subsec:uncount}

Suppose that $Y$ is an uncountable standard Borel space with a Borel probability measure $\rho$ and $V$ is a graphon on $Y$.
A rough strategy to define $V$-biregular blowup is the same as above, i.e., replace each point by a copy of a unit interval and glue together elements of $\bir$ according to values of $V$.
However, we  need to be more careful in this case to preserve measurability.

Let $X=Y\times [0,1]$ and $\mu=\rho\times \lambda$ be the product measure.
One can think of $\mu$ as a collection of measures $\{\lambda_y\}_{y\in Y}$, where $\lambda_y$ is the Lebesgue measure on the strip $\{y\}\times [0,1]$ such that 
$$\mu(A)=\int_{Y} \lambda_y(A) \ d\rho(y)$$
holds for every Borel set $A\subseteq X$.
Let 
$$\Omega:Y\times Y\to \bir$$
be a Borel map that satisfies $\Omega(y,z)=\Omega(z,y)$ and $\Omega(y,z)\in \bir_{V(y,z)}$.
A \emph{$V$-biregular blowup} that is given by $\Omega$ is a graphon $W_\Omega$ on $X$ defined as 
$$W_\Omega((y,r),(z,s))=\Omega(y,z)(r,s).$$
Let $\eta=\{\{y\}\times [0,1]\}_{y\in Y}$ be a measurable partition of $X$ and $\mathcal{C}$ be the sub-$\sigma$-algebra generated by $\eta$.
It follows from the construction that the following condition, a measurable analogue of equitable partition, is satisfied
$$\deg_W((y,r),\{z\}\times [0,1])=\int_{[0,1]} W_\Omega((y,r),(z,{-})) \ d\lambda_z=V(y,z)$$
holds for every $y,z\in Y$ and $\lambda_y$-almost every $(y,r)\in \{y\}\times[0,1]$.
It is straightforward to check that this condition implies that $\mathcal{C}$ is $W_\Omega$-invariant and $W_\Omega/\mathcal{C}=V$.
Consequently all $V$-biregular blowups of $V$ are pairwise fractionally isomorphic.

\subsection{Reversed direction}

We briefly sketch why the above construction describes all the examples without going into technical details.

Let $W$ be a graphon on $X$ and $\mathcal{C}$ be a $W$-invariant sub-$\sigma$-algebra.
Up to a small technical nuance, it follows from the Measure Disintegration Theorem, see~\cite[Exercise~17.35]{kechris1995classical}, that there is a standard Borel space $Y$ with a Borel probability measure $\rho$ and an isomorphism between $(X,\mu)$ and $(Y\times [0,1],\rho\times \lambda)$ such that $\mathcal{C}$ is exactly the sub-$\sigma$-algebra generated by the preimage of the measurable partition $\eta=\{\{y\}\times [0,1]\}_{y\in Y}$ under this isomorphism.
Therefore, we may abuse the notation and assume that $W$ is a graphon on $Y\times [0,1]$ and  $V=W/\mathcal{C}$ is a graphon on $Y$.
Define
$$\Omega(y,z)=W\upharpoonright \left(\{y\}\times [0,1]\right)\times \left(\{z\}\times [0,1]\right).$$
It follows that $\Omega$ is a Borel map and one can show that the condition that $\mathcal{C}$ is $W$-invariant implies $\Omega(y,z)\in \bir_{V(y,z)}$ for $(\rho\times \rho)$-almost every $(y,z)\in Y\times Y$.

Now by (5) in Theorem~\ref{th:main}, if $W$ and $U$ are fractionally isomorphic, then they are $V$-biregular blowups, where $V=W/\mathcal{C}(W)=W/\mathcal{C}(U)$.

\section{Further remarks and problems}\label{sec:problems}

A direct consequence of Theorem~\ref{th:main} is that the assignment
$$W\mapsto W_{\mathcal{C}(W)}$$
is a well defined map from $\widetilde{\mathcal{W}_0}$ to $\widetilde{\mathcal{W}}_0$.
We denote the range of the map as $\mathcal{F}\subseteq \widetilde{\mathcal{W}}_0$ and call elements of $\mathcal{F}$ {\it fraction-free} graphons.
It follows from (3) in Theorem~\ref{th:main} that the restriction of the equivalence relation induced by fractional isomorphism to $\mathcal{F}$ is equal to weak isomorphism. 
Finally, it follows Corollary~\ref{cor:1 implies 2} that $W\mapsto \nu_W$ is a cut-distance continuous map when the set of all Borel probability measures on $\mathbb{M}$, $\mathscr{P}(\mathbb{M})$, is endowed with the weak* topology.
Therefore those DIDM that correspond to graphons form a closed subset of $\mathscr{P}(\mathbb{M})$.

\begin{question}
Is $W\mapsto W_{\mathcal{C}(W)}$ cut-distance continuous?
\end{question}

This is equivalent with $\mathcal{F}$ being closed.
Suppose that $\mathcal{F}$ is closed, $U_n\to_{\delta_\Box} U$ and put $V_n=(U_n)_{\mathcal{C}(U_n)}$.
By compactness of cut-distance and our assumption, we may assume that $V_n \to_{\delta_\Box} V\in \mathcal{F}$.
Since fractional isomorphism is a closed equivalence relation we have that $V$ and $U$ are fractionally isomorphic.
By (3) and (5) in Theorem~\ref{th:main}, we deduce that $V_{\mathcal{C}(V)}$ is weakly isomorphic to $U_{\mathcal{C}(U)}$.
However, $V_{\mathcal{C}(V)}=V$ and that gives immediately $U_{\mathcal{C}(U)}=V$ in $\widetilde{\mathcal{W}}_0$.
Reversed implication is trivial.

\begin{question}Let $W$ and $U$ be fractionally isomorphic graphons.
Is it possible to find sequences $\{G_n\}_{n\in \mathbb{N}}$ and $\{H_n\}_{n\in \mathbb{N}}$ of finite graphs such that $G_n$ is fractionally isomorphic to $H_n$ for each $n\in \mathbb{N}$ and
$$G_n\to_{\delta_\Box} W \ \text{ and } \ H_n\to_{\delta_\Box} U?$$
\end{question}
A positive answer to this question combined with the observation that fractional isomorphism is a closed equivalence relation would provide a new characterization in Theorem~\ref{th:main}.

\section{Structure of the proof}\label{sec:strategy}

We summarize the structure of the proof of Theorem~\ref{th:main}.
We note that it is more suitable to work with general integral kernels (non-symmetric functions) rather than graphons.
\begin{itemize}
	\item (1) $\Rightarrow$ (2): we define a collection of continuous functions $\mathcal{T}\subseteq C(\mathbb{M},\mathbb{R})$ that corresponds in a certain sense to tree densities and separates points of $\mathbb{M}$ (Section~\ref{sec:Tree functions}), then we use a version of Stone-Weierstrass's Theorem (Corollary~\ref{cor:SepMeasures}),
	\item (2) $\Rightarrow$ (3): we define an integral kernel ${\bf U}[\nu]$ for every DIDM and show that ${\bf U}[\nu_W]$ and $W/\mathcal{C}(W)$ are isomorphic for every graphon $W$ (Section~\ref{sec:DIDD}),
	\item (3) $\Rightarrow$ (4): we show that $\mathbb{E}({-}|\mathcal{C}(W))\circ T_W=T_{W_{\mathcal{C}(W)}}\circ\mathbb{E}({-}|\mathcal{C}(W))$ and that isomorphic graphons are intertwined by a Markov operator (Section~\ref{sec:Basic} and Appendix~\ref{App E}),
	\item (4) $\Rightarrow$ (5): we observe that (4) implies $T_W\circ (S\circ S^*)=(S\circ S^*)\circ T_W$ (similarly for $U$) and use the Mean Ergodic Theorem (Theorem~\ref{th:Mean Ergodic}) to show that $\frac{1}{n}\sum_{k\in [n]} (S^*S)^k$ converges to a Markov projection; then we exploit the duality between Markov projections and relatively complete sub-$\sigma$-algebras (Appendix~\ref{App D}),
	\item (5) $\Rightarrow$ (1): tree densities are preserved when taking a conditional expectation given invariant sub-$\sigma$-algebras (Section~\ref{sec:Tree functions}).
\end{itemize}

\section{Subalgebras}\label{sec:Basic}

In this section we prove basic statements about invariant sub-$\sigma$-algebras, conditional expectations and quotients of graphons, and define the minimum $W$-invariant sub-$\sigma$-algebra $\mathcal{C}(W)$ via a canonical sequence of sub-$\sigma$-algebras $\left\{\mathcal{C}^W_n\right\}_{n\in \mathbb{N}}$.

Recall that $(X,\mathcal{B})$ is a standard Borel space and $\mu$ is a Borel probability measure on $X$, see Appendix~\ref{App A}.
The $L^2$-spaces are real and we denote the scalar product as $\langle{-},{-}\rangle$.
For $V\subseteq L^2(X,\mu)$ we let $V^\bot$ be the orthogonal complement of $V$.
We write ${\bf 1}_A$ for the characteristic function of $A\subseteq X$.
If $\mathcal{C}$ is (relatively complete) sub-$\sigma$-algebra of $\mathcal{B}$, then it is a standard fact that the linear hull of $\{{\bf 1}_A\}_{A\in \mathcal{C}}$ is dense in $L^2(X,\mathcal{C},\mu)$, see the corresponding definitions below.

If $f$ and $g$ are measurable functions defined on some measure space $Y$, then we abuse the notation and write $f=g$ for equality almost everywhere.
It is always clear from the context what type of equality we mean.

\subsection{Kernels}

An {\it integral kernel on $X$} is a $\left(\mathcal{B}\times \mathcal{B}\right)$-measurable map
$$W:X\times X\to [0,1].$$
The corresponding integral operator $T_W:L^2(X,\mu)\to L^2(X,\mu)$ defined as
$$T_W(f)(x)=\int_X W(x,y)f(y) \ d\mu(y)$$
is a  well-defined Hilbert-Schmidt operator (see \cite[Chapter 4, Exercise 15]{Rud2}).
We consider integral kernels $W$ and $U$ on $X$ to be the same if $T_W=T_U$.
It is a standard fact that this is equivalent with $W(x,y)=U(x,y)$ for $\left(\mu\times \mu\right)$-almost every $(x,y)\in X\times X$.
In other words, $W$ and $U$ are the same as elements of $L^\infty(X\times X,\mu\times \mu)$.
We say that an integral kernel $W$ is a {\it graphon (on $X$)} if $W(x,y)=W(y,x)$ for $\left(\mu\times \mu\right)$-almost every $(x,y)\in X\times X$.

\begin{claim}\label{cl:Self Adjoint=Graphon}
Let $W$ be an integral kernel on $X$.
Then $T_W$ is self-adjoint if and only if $W$ is a graphon.
\end{claim}

For a closed linear subspace $V\subseteq L^2(X,\mu)$ we denote as $P_V$ the orthogonal projection onto $V$.
We say that a subspace $V$ is \emph{$W$-invariant}, where $W$ is an integral kernel, if $T_W(V)\subseteq V$.
The following characterization of invariant subspaces for graphons is a standard application of the fact that $T_W$ is a compact operator.

\begin{proposition}\label{pr:Invariant Subspace}
Let $W$ be a graphon and $V\subseteq L^2(X,\mu)$ be a closed linear subspace.
Then the following are equivalent
\begin{enumerate}
	\item $V$ is $W$-invariant,
	\item there is an orthonormal basis of $V$ made of eigenvectors of $T_W$,
	\item $T_W$ commutes with the projection $P_V$,
	\item $T_W(V^\bot)\subseteq V^\bot$.
\end{enumerate}
\end{proposition}

\subsection{Conditional Expectation and Invariant Subspaces}

\begin{definition}[Relative complete sub-$\sigma$-algebra]
We say that $\mathcal{C}\subseteq \mathcal{B}$ is a \emph{$\mu$-relatively complete sub-$\sigma$-algebra} of $\mathcal{B}$ if it is a sub-$\sigma$-algebra and $Z\in \mathcal{C}$ whenever there is $Z_0\in \mathcal{C}$ such that $\mu(Z\triangle Z_0)=0$.
We define $\Theta_\mu$ as the set of all $\mu$-relatively complete sub-$\sigma$-algebras of $\mathcal{B}$.
\end{definition}

Since the measure $\mu$ is always fixed we say simply relatively complete sub-$\sigma$-algebra.

\begin{claim}\label{cl:Intersections of rel compl algebras}
Let $\Phi$ be a non-empty family of relatively complete sub-$\sigma$-algebras.
Then 
$$\left\{Z\in \mathcal{B}:\forall \mathcal{C}\in \Phi, \  Z\in \mathcal{C} \right\}\in \Theta_\mu.$$
\end{claim}

As a direct consequence we have that every $\mathcal{X}\subseteq \mathcal{C}$ generates a unique relatively complete sub-$\sigma$-algebra that we denote as $\left\langle \mathcal{X}\right\rangle$.

Given $\mathcal{C}\in \Theta_\mu$ we define $L^2(X,\mathcal{C},\mu)$ to be the collection of all functions in $L^2(X,\mu)$ that are $\mathcal{C}$-measurable.
A standard fact about conditional expectation, see  Theorem~\ref{th:ConditionalExp}, yields the following.

\begin{claim}\label{cl:Basic 0}
Let $\mathcal{C}\in \Theta_\mu$.
Then $L^2(X,\mathcal{C},\mu)$ is a closed linear subspace and 
$$\mathbb{E}\left({-}|\mathcal{C}\right):L^2(X,\mu)\to L^2(X,\mu)$$
is the orthogonal projection onto $L^2(X,\mathcal{C},\mu)$.
\end{claim}

In the introduction we defined for a graphon $W$ and a $W$-invariant algebra $\mathcal{C}\in \Theta_\mu$ a graphon $W_\mathcal{C}$ as the conditional expectation of $W$ given $\mathcal{C}\times \mathcal{C}$.
Here, we slightly abuse the notation and define $W_\mathcal{C}$ as the conditional expectation of $W$ given $\mathcal{B}\times \mathcal{C}$, i.e.,
$$W_{\mathcal{C}}=\mathbb{E}\left(W|\mathcal{B}\times \mathcal{C}\right),$$
for every integral kernel $W$ and any $\mathcal{C}\in \Theta_\mu$.
We show in Claim~\ref{cl:Basic 1} that for graphons the assumption that the algebra is invariant implies that these definitions are the same.

\begin{claim}\label{cl:How quotient acts}
Let $\mathcal{C}\in \Theta_\mu$.
Then $T_{W_\mathcal{C}}=T_W\circ \mathbb{E}({-}|\mathcal{C})$.
In particular, $T_W\upharpoonright L^2(X,\mathcal{C},\mu)=T_{W_\mathcal{C}}\upharpoonright L^2(X,\mathcal{C},\mu)$.
\end{claim}
\begin{proof}
Let $A\in \mathcal{C}$ and $B\in \mathcal{B}$.
Then we have
\begin{equation*}
\begin{split}
\left\langle T_{W_{\mathcal{C}}}({\bf 1}_A),{\bf 1}_B\right\rangle = & \ \int_{B\times A} W_\mathcal{C} \ d(\mu \times \mu) = \ \int_{B\times A} W \ d(\mu\times \mu) \\
= & \ \left\langle T_W({\bf 1}_A),{\bf 1}_B\right \rangle = \left\langle \left(T_W\circ \mathbb{E}({-}|\mathcal{C})\right)({\bf 1}_A),{\bf 1}_B\right \rangle,
\end{split}
\end{equation*}
where we used Theorem~\ref{th:ConditionalExp}~(3) in the second equality.
Since linear hulls of $\{{\bf 1}_A\}_{A\in \mathcal{C}}$ and $\{{\bf 1}_B\}_{B\in \mathcal{B}}$ are dense in $L^2(X,\mathcal{C},\mu)$ and $L^2(X,\mu)$, respectively, we get that the claim holds for every $f\in L^2(X,\mathcal{C},\mu)$.

Let $f\in L^2(X,\mathcal{C},\mu)^\bot$ and $B\in \mathcal{B}$.
Define $F(x,y)=f(y)$ and note that $\mathbb{E}(F|\mathcal{B}\times \mathcal{C})=0$ by Fubini's Theorem.
We have
\begin{equation*}
\begin{split}
\left\langle T_{W_{\mathcal{C}}}(f),{\bf 1}_B\right\rangle = & \ \int_{B\times X} W_\mathcal{C}(x,y)f(y) \ d(\mu \times \mu)(x,y) = \int_{B\times X} W_\mathcal{C}(x,y)F(x,y) \ d(\mu \times \mu)(x,y)\\
= & \ \int_{B\times X} W(x,y)\mathbb{E}(F|\mathcal{B}\times \mathcal{C})(x,y) \ d(\mu\times \mu)(x,y)=0,
\end{split}
\end{equation*}
where we used Theorem~\ref{th:ConditionalExp}~(2) in the third equality.
This implies that $T_{W_\mathcal{C}}(f)=0$ and the proof is finished.
\end{proof}

We say that $\mathcal{C}\in \Theta_\mu$ is {\it $W$-invariant} if $L^2(X,\mathcal{C},\mu)$ is $W$-invariant, i.e., if $T_W(L^2(X,\mathcal{C},\mu))\subseteq L^2(X,\mathcal{C},\mu)$.
Equivalently by Claim~\ref{cl:How quotient acts}, we have
$$ T_{W_\mathcal{C}}\circ \mathbb{E}({-}|\mathcal{C})=T_W\circ \mathbb{E}({-}|\mathcal{C})=\mathbb{E}({-}|\mathcal{C})\circ T_W\circ \mathbb{E}({-}|\mathcal{C})=\mathbb{E}({-}|\mathcal{C})\circ T_{W_\mathcal{C}},$$
i.e., $T_{W_\mathcal{C}}$ commutes with $\mathbb{E}({-}|\mathcal{C})$.

\begin{claim}\label{cl:Basic 1}
Let $\mathcal{C}\in \Theta_\mu$ be $W$-invariant.
Then $W_\mathcal{C}=\mathbb{E}(W|\mathcal{B}\times \mathcal{C})=\mathbb{E}(W|\mathcal{C}\times \mathcal{C})$.
Moreover, if $W$ is a graphon, then so is $W_\mathcal{C}$.
\end{claim}
\begin{proof}
Let $U=\mathbb{E}(W|\mathcal{C}\times \mathcal{C})$ and $A,B\in \mathcal{C}$.
We have 
\begin{equation*}
\begin{split}
\left\langle T_{W_{\mathcal{C}}}({\bf 1}_A),{\bf 1}_B\right\rangle = & \ \int_{B\times A} W_\mathcal{C} \ d(\mu \times \mu) = \int_{B\times A} W \ d(\mu \times \mu)\\
= & \ \int_{B\times A} \mathbb{E}(W|\mathcal{C}\times \mathcal{C}) \ d(\mu\times \mu)=\left\langle T_{U}({\bf 1}_A),{\bf 1}_B\right\rangle
\end{split}
\end{equation*}
by Theorem~\ref{th:ConditionalExp}~(3).
The assumption that $\mathcal{C}$ is $W$-invariant implies that $T_{W_\mathcal{C}}(f)=T_U(f)$ for every $f\in L^2(X,\mathcal{C},\mu)$.
It follows from Claim~\ref{cl:How quotient acts} that $T_{W_{\mathcal{C}}}(f)=0$ whenever $f\in L^2(X,\mathcal{C},\mu)^\bot$ and it is easy to see that the same argument as in the proof of Claim~\ref{cl:How quotient acts} shows that the same holds for $T_U$.
Then we have $T_{W_\mathcal{C}}=T_U$ and consequently $W_{\mathcal{C}}=U$.
The additional part follows easily by Claim~\ref{cl:Self Adjoint=Graphon}.
\end{proof}

Taking conditional expectation can be reformulated in the language of quotient spaces.
First we recall Theorem~\ref{th:Quotients and Markov}.
For every $\mathcal{C}\in \Theta_\mu$ there is a standard Borel space $(X/\mathcal{C},\mathcal{C}')$, a probability measure $\mu/\mathcal{C}\in \mathcal{P}(X/\mathcal{C})$ and a Borel map $q_\mathcal{C}:X\to X/\mathcal{C}$ such that $\mu/\mathcal{C}$ is the push-forward of $\mu$ via $q_\mathcal{C}$.
Moreover there is a unique linear isometry 
$$I_\mathcal{C}:L^2(X/\mathcal{C},\mu/\mathcal{C})\to L^2(X,\mu)$$
defined as
$$I_\mathcal{C}(f)(x)=f(q_\mathcal{C}(x))$$
that is a Markov operator onto $L^2(X,\mathcal{C},\mu)$.
If we write $S_\mathcal{C}$ for the adjoint of $I_\mathcal{C}$, then $S_\mathcal{C}$ is a Markov operator, $S_\mathcal{C}\upharpoonright L^2(X,\mathcal{C},\mu)$ is an isometrical isomorphism and $S_\mathcal{C}=S_{\mathcal{C}}\circ \mathbb{E}({-}|\mathcal{C})$.
It follows that $S_\mathcal{C}\circ I_\mathcal{C}$ is the identity on $L^2(X/\mathcal{C},\mu/\mathcal{C})$ and $I_\mathcal{C}\circ S_\mathcal{C}$ is equal to $\mathbb{E}({-}|\mathcal{C})$.

\begin{definition}
Let $\mathcal{C}\in \Theta_\mu$ be $W$-invariant.
We define $W/\mathcal{C}=S_{\mathcal{C}\times \mathcal{C}}(W_\mathcal{C})$.
\end{definition}

Formally, $W/\mathcal{C}$ is defined on the space $(X\times X)/(\mathcal{C}\times \mathcal{C})$ but it can be easily verified that there is a measure preserving bijection
$$i:(X\times X)/(\mathcal{C}\times \mathcal{C})\to (X/ \mathcal{C})\times (X/\mathcal{C})$$
such that $(i\circ q_{\mathcal{C}\times \mathcal{C}})(x,y)=(q_{\mathcal{C}}(x),q_{\mathcal{C}}(y))$ for $(\mu\times \mu)$-almost every $(x,y)\in X\times X$.
Therefore, we abuse the notation and assume that $W/\mathcal{C}$ is defined on $X/\mathcal{C}\times X/\mathcal{C}$.
Consequently by Claim~\ref{cl:Basic 1}, we have $I_{\mathcal{C}\times \mathcal{C}}(W/\mathcal{C})=W_\mathcal{C}$ and
$$W_\mathcal{C}(x,y)=(W/\mathcal{C})(q_{\mathcal{C}}(x),q_{\mathcal{C}}(y))$$
for $(\mu\times \mu)$-almost every $(x,y)\in X\times X$.

\begin{proposition}\label{pr:Quotient and Invariant algebra}
Let $W$ be an integral kernel and $\mathcal{C}\in \Theta_\mu$ be $W$-invariant.
Then
\begin{enumerate}
	\item[(i)] if $W$ is a graphon, then $W/\mathcal{C}$ is a graphon. Furthermore, $W_\mathcal{C}$ and $W/\mathcal{C}$ are weakly isomorphic,
	\item[(ii)] $T_{W/\mathcal{C}}\circ S_\mathcal{C}=S_\mathcal{C}\circ T_{W_\mathcal{C}}$,
	\item[(iii)] if $W$ is a graphon, then we have $T_{W/\mathcal{C}}\circ S_\mathcal{C}=S_\mathcal{C}\circ T_{W}$.
\end{enumerate}
\end{proposition}
\begin{proof}
{\bf (i)} It follows from the remark before this proposition that $W_\mathcal{C}$ is a pull-back of $W/\mathcal{C}$.
This implies easily both claims in (i).

{\bf (ii)} If $f\in L^2(X,\mathcal{C},\mu)^\bot$, then the equality clearly holds.
Suppose that $f_0,f_1\in L^2(X,\mathcal{C},\mu)$.
By the definition, we find $h_0,h_1\in L^2(X/\mathcal{C},\mu/\mathcal{C})$ such that $I_\mathcal{C}(h_i)=f_i$ and $S_\mathcal{C}(f_i)=h_i$ for $i\in \{0,1\}$.
Then we have
\begin{equation*}
\begin{split}
\left\langle \left(T_{W/\mathcal{C}}\circ S_\mathcal{C}\right)(f_0),h_1\right\rangle = & \ \left\langle T_{W/\mathcal{C}}(h_0),h_1\right\rangle \\
= & \ \int_{(X/\mathcal{C})\times (X/\mathcal{C})} h_1(r)(W/\mathcal{C})(r,s)h_0(s) \ d((\mu/\mathcal{C})\times (\mu/\mathcal{C}))(r,s) \\
= & \ \int_{X\times X} f_1(x)W_{\mathcal{C}}(x,y)f_0(y) \ d(\mu\times \mu)(x,y)=\left\langle T_{W_\mathcal{C}}(f_0),f_1\right\rangle \\
= & \ \left\langle T_{W_\mathcal{C}}(f_0),I_\mathcal{C}(h_1)\right\rangle=\left\langle \left(S_\mathcal{C}\circ T_{W_\mathcal{C}}\right)(f_0),h_1\right\rangle
\end{split}
\end{equation*}
and the claim follows.

{\bf (iii)} Proposition~\ref{pr:Invariant Subspace} implies that $T_W$ commutes with $\mathbb{E}({-}|\mathcal{C})$.
By (ii) and Claim~\ref{cl:How quotient acts}, we have
$$T_{W/\mathcal{C}}\circ S_\mathcal{C}=S_\mathcal{C}\circ T_{W_\mathcal{C}}=S_\mathcal{C}\circ T_W\circ \mathbb{E}({-}|\mathcal{C})=S_\mathcal{C}\circ \mathbb{E}({-}|\mathcal{C})\circ T_W=S_\mathcal{C}\circ T_W$$
and the proof is finished.
\end{proof}

\subsection{The minimum invariant sub-$\sigma$-algebra}

Let $W$ be an integral kernel on $X$.
We show in this section that there is the minimum $W$-invariant relatively complete sub-$\sigma$-algebra and that it admits a canonical description.
First we need to introduce some auxiliary notion.

\begin{definition}
Let $\mathcal{D},\mathcal{E}\in \Theta_\mu$.
We say that $(\mathcal{D},\mathcal{E})$ is a \emph{$W$-invariant pair} if
$$T_W(L^2(X,\mathcal{D},\mu))\subseteq L^2(X,\mathcal{E},\mu).$$
\end{definition}

Note that $\mathcal{C}\in \Theta_\mu$ is $W$-invariant if and only if $(\mathcal{C},\mathcal{C})$ is a $W$-invariant pair.
Given $\mathcal{C}\in \Theta_\mu$ define $\Phi$ to be the collection of $\mathcal{D}\in \Theta_\mu$ such that $(\mathcal{C},\mathcal{D})$ is a $W$-invariant pair.
Then $\Phi$ is non-empty because $\mathcal{B}\in \Phi$.
By Claim~\ref{cl:Intersections of rel compl algebras}, we have 
$$m(\mathcal{C})=\left\{Z\in \mathcal{B}:\forall \mathcal{D}\in \Phi, \ Z\in \mathcal{D}\right\}\in \Theta_\mu.$$
The following is straightforward.

\begin{claim}\label{cl:invariantPair}
Let $\mathcal{C}\in \Theta_\mu$.
Then $(\mathcal{C},m(\mathcal{C}))$ is a $W$-invariant pair.
\end{claim}

\begin{definition}[Canonical sequence $\left\{\mathcal{C}^W_n\right\}_{n\in \mathbb{N}}$]
Define $\mathcal{C}^W_0=\langle\{\emptyset,X\}\rangle$ and inductively $\mathcal{C}^W_{n+1}=m\left(\mathcal{C}^W_n\right)$.
Furthermore, we define
$$\mathcal{C}(W)=\left\langle\bigcup_{n\in \mathbb{N}}\mathcal{C}^W_n\right\rangle.$$
\end{definition}

\begin{proposition}
Let $W$ be an integral kernel.
Then $\mathcal{C}(W)$ is the minimum $W$-invariant relatively complete sub-$\sigma$-algebra of $\mathcal{B}$. 
\end{proposition}
\begin{proof}
Suppose that $\mathcal{C}\in \Theta_\mu$ is $W$-invariant.
Then we have trivially $\mathcal{C}^W_0\subseteq \mathcal{C}$ and by induction $\mathcal{C}^W_n\subseteq \mathcal{C}$ for every $n\in\mathbb{N}$.
This shows $\mathcal{C}(W)\subseteq \mathcal{C}$.

It remains to show that $\mathcal{C}(W)$ is $W$-invariant.
First note that $\bigcup_{n\in \mathbb{N}} \mathcal{C}^W_n$ is an algebra (not necessarily $\sigma$-algebra) that generates $\mathcal{C}(W)$.
By \cite[Exercise~17.43]{kechris1995classical}, we can find for each $A\in \mathcal{C}(W)$ a sequence $A_n\in \mathcal{C}^W_n$ such that ${\bf 1}_{A_n}\to {\bf 1}_A$ in $L^2(X,\mu)$.
By continuity of $T_W$, we have $T_W({\bf 1}_{A_n})\to T_W({\bf 1}_A)$ in $L^2(X,\mu)$ and, by Claim~\ref{cl:invariantPair}, we have $T_W({\bf 1}_{A_n})\in L^2\left(X,\mathcal{C}^W_{n+1},\mu\right)\subseteq L^2(X,\mathcal{C}(W),\mu)$.
Since $L^2(X,\mathcal{C}(W),\mu)$ is closed, by Claim~\ref{cl:Basic 0}, we have $T_W({\bf 1}_A)\in L^2(X,\mathcal{C}(W),\mu)$.
Since the linear hull of $\left\{{\bf 1}_A\right\}_{A\in \mathcal{C}(W)}$ is dense in $L^2(X,\mathcal{C}(W),\mu)$ and $T_W$ is linear and continuous we conclude that $\mathcal{C}(W)$ is $W$-invariant.
\end{proof}

\section{Distributions on iterated degree measures}\label{sec:DIDD}

In this section we define the compact metric space $\mathbb{M}$ whose elements are \emph{iterated degree measures}.
This definition is independent of $\mathcal{W}_0$.
We assign to a graphon $W$ on $X$ a Borel map $i_W:X\to  \mathbb{M}$ and a Borel probability measure $\nu_W$ on $\mathbb{M}$ that encodes the canonical sequence $\{\mathcal{C}^W_n\}_{n\in \mathbb{N}}$.
These measures are called \emph{distributions on iterated degree measures, DIDM}.
Lastly, we show that every DIDM $\nu$ encodes an integral kernel ${\bf U}[\nu]$ on $\mathbb{M}$ such that $W/\mathcal{C}(W)$ is isomorphic to ${\bf U}[\nu_W]$ for every graphon $W$.

\subsection{The Space $\mathbb{M}$}

For a compact metric space $K$ we denote as $\mathscr{M}_{\le 1}(K)$ the set of all Borel measures on $K$ of total mass at most $1$.
Moreover, we put $\mathscr{P}(K)$ for the set of all Borel probability measures on $K$, i.e., distributions on $K$, and we denote as $C(K,\mathbb{R})$ the space of all real-valued continuous functions on $K$.
It is a standard fact from functional analysis that $\mathscr{M}_{\le 1}(K)$ and $\mathscr{P}(K)$ are compact and metrizable when endowed with the weak* topology, see Appendix~\ref{App B}.

\begin{definition}
Let $P^0=\{\star\}$ be the one-point space and define inductively
$$\mathbb{M}_n=\prod_{i\le n} P^i \ \operatorname{and} \ P^{n+1}=\mathscr{M}_{\le 1}\left(\mathbb{M}_n\right)$$
for every $n\in \mathbb{N}$.
We put $\mathbb{M}=\mathbb{M}_\infty=\prod_{n\in \mathbb{N}} P^n$ and denote as $p_{n,k}:\mathbb{M}_k\to\mathbb{M}_n$ the canonical projection, where $n\le k\le \infty$.
\end{definition}

It is an easy consequence of the discussion above together with Tychonoff's Theorem, see~\cite[Theorem A3]{Rud2}, that $\mathbb{M}$ is a compact metric space.

A particularly interesting subspace of $\mathbb{M}$ consists of coherent sequences of measures.
Namely, define
$$\mathbb{P}=\left\{\alpha\in \mathbb{M}: \forall n\in \mathbb{N} \ \alpha(n+1)=(p_{n,n+1})_*\alpha(n+2)\right\},$$
where $(p_{n,n+1})_*\alpha(n+2)\in \mathscr{M}_{\le 1}(\mathbb{M}_{n})$ denotes the push-forward of $\alpha(n+2)\in \mathscr{M}_{\le 1}(\mathbb{M}_{n+1})$ via $p_{n,n+1}$, see Appendix~\ref{App A} for definition.
It follows from Kolmogorov's Existence Theorem \cite[Theorem~36.1]{billingsley} that for every $\alpha\in \mathbb{P}$ there is a unique $\mu_\alpha\in \mathscr{M}_{\le 1}(\mathbb{M})$ such that 
\begin{equation*}\label{eq:mu_alpha}
(p_{n,\infty})_*\mu_\alpha=\alpha(n+1).
\end{equation*}
for every $n\in \mathbb{N}$.
In fact, we have the following uniform version.

\begin{claim}\label{cl:uniformKolmogorov}
The set $\mathbb{P}$ is closed in $\mathbb{M}$ and the map $\alpha\mapsto \mu_\alpha$ that satisfies
$$(p_{n,\infty})_*\mu_\alpha=\alpha(n+1)$$
for every $n\in \mathbb{N}$ is a continuous map from $\mathbb{P}$ to $\mathscr{M}_{\le 1}(\mathbb{M})$.
\end{claim}
\begin{proof}
Let $\left\{\alpha_k\right\}_{k\in \mathbb{N}}\subseteq \mathbb{P}$, $\alpha\in \mathbb{M}$ be such that $\alpha_k\to \alpha$ and $n\in \mathbb{N}$.
By the definition, we have $\alpha_{k}(n+2)\to \alpha(n+2)$ in $\mathscr{M}_{\le 1}(\mathbb{M}_{n+1})$ and $(p_{n,n+1})_*\alpha_k(n+2)=\alpha_k(n+1)\to \alpha(n+1)$ in $\mathscr{M}_{\le 1}(\mathbb{M}_n)$.
However, this implies
\begin{equation*}
\begin{split}
\int_{\mathbb{M}_n} f \ d\alpha_k(n+1)= & \ \int_{\mathbb{M}_{n+1}} f\circ p_{n,n+1} \ d\alpha_k(n+2) \\
\to & \  \int_{\mathbb{M}_{n+1}} f\circ p_{n,n+1} \ d\alpha(n+2)=\int_{\mathbb{M}_n} f \ d(p_{n,n+1})_*\alpha(n+2)
\end{split}
\end{equation*}
for every $f\in C(\mathbb{M}_{n},\mathbb{R})$.
This shows that $\alpha(n+1)=(p_{n,n+1})_*\alpha(n+2)$ and consequently that $\alpha\in \mathbb{P}$.

It follows from Theorem~\ref{th:StoneWeierstrass} that
$$\mathcal{A}=\bigcup_{n\in \mathbb{N}} C(\mathbb{M}_n,\mathbb{R})\circ p_{n,\infty}$$
is uniformly dense in $C(\mathbb{M},\mathbb{R})$.
Let $\alpha_k,\alpha\in \mathbb{P}$ for every $k\in \mathbb{N}$ such that $\alpha_k\to \alpha$ in $\mathbb{M}$ (or equivalently in $\mathbb{P}$).
This means by definition that $(p_{n,\infty})_*\mu_{\alpha_k}=\alpha_k(n+1)\to \alpha(n+1)=(p_{n,\infty})_*\mu_\alpha$ for every $n\in \mathbb{N}$.
Then we have
\begin{equation*}
\begin{split}
\int_\mathbb{M} f\circ p_{n,\infty} \ d\mu_{\alpha_k}= & \ \int_{\mathbb{M}_n} f \ d(p_{n,\infty})_*\mu_{\alpha_k} \\
\to & \  \int_{\mathbb{M}_n} f \ d(p_{n,\infty})_*\mu_{\alpha}=\int_\mathbb{M} f\circ p_{n,\infty} \ d\mu_\alpha
\end{split}
\end{equation*}
for every $f\in C(\mathbb{M}_n,\mathbb{R})$.
It follows from the the uniform density of $\mathcal{A}$ that $\mu_{\alpha_k}\to \mu_\alpha$ in $\mathscr{M}_{\le 1}(\mathbb{M})$.
\end{proof}

Finally we are ready to state the main definition of this section.
Note that in the definition, (2) makes sense by (1).

\begin{definition} We say that $\nu \in \mathscr{P}(\mathbb{M})$ is a \emph{distribution on iterated degree measures, DIDM}, if 
\begin{enumerate}
	\item $\nu(\mathbb{P})=1$,
	\item $\mu_\alpha$ is absolutely continuous with respect to $\nu$ with the corresponding Radon--Nikodym derivative satisfying $0\le \frac{d\mu_\alpha}{d\nu}\le 1$ for $\nu$-almost every $\alpha\in \mathbb{M}$.
\end{enumerate}
\end{definition}

\subsection{From Kernels to DIDM}

For a given integral kernel $W$ on $X$ we define inductively a map $i_W:X\to \mathbb{M}$ and show that $\nu_W$, the push-forward of $\mu$ via $i_W$, is a DIDM.
Compare the definition of $i_W$ with the informal definition given in the introduction.
Moreover, we show that $\mathcal{C}(W)$ is the minimum relatively complete sub-$\sigma$-algebra that makes $i_W$ measurable.

\begin{definition}
Let $(X,\mathcal{B})$ be a standard Borel space and $W$ be an integral kernel on $X$.
We define $i_{W,0}:X\to \mathbb{M}_0=\{\star\}$ to be the constant map.
Inductively, we define $i_{W,n+1}:X\to \mathbb{M}_{n+1}$   such that
\begin{itemize}
	\item[(a)] $i_{W,n+1}(x)(j)=i_{W,n}(x)(j)$, for every $j\le n$ and
	\item[(b)] $i_{W,n+1}(x)(n+1)(A)=\int_{i_{W,n}^{-1}(A)} W(x,{-}) \ d\mu$, whenever $A\subseteq \mathbb{M}_n$ is a Borel set.
\end{itemize}
Denote as
$$i_W:X\to \mathbb{M}$$
the unique map defined as $i_W(x)(n)=i_{W,n}(x)(n)$.
Finally, let $\nu_W$ to be the push-forward of $\mu$ via $i_W$.
\end{definition}

To make sure that we can proceed with the inductive construction and that $\nu_W$ is well-defined we need to show that $i_{W,n}$ is a measurable map for every $n\in \mathbb{N}$.
In fact, we show that $\mathcal{C}^W_n$ is the minimum relatively complete sub-$\sigma$-algebra that makes $i_{W,n}$ measurable.

For each $n\in \mathbb{N}$ denote as $\mathcal{B}(\mathbb{M}_n)$ the Borel $\sigma$-algebra of $\mathbb{M}_n$.
First we need a claim that we use in our inductive arguments.

\begin{claim}\label{cl:3-Basic 1}
Let $n\in \mathbb{N}$ and suppose that $i_{W,n}$ is measurable.
Then
$$\int_{\mathbb{M}_n} f \ d\left(i_{W,n+1}(x)(n+1)\right)=\int_X W(x,y) (f\circ i_{W,n})(y) \ d\mu(y)$$
for every bounded Borel function $f:\mathbb{M}_n\to \mathbb{R}$ and every $x\in X$.
\end{claim}
\begin{proof}
This a straightforward consequence of (b) from the definition of $i_{W,n+1}$.
\end{proof}

\begin{proposition}\label{pr:Description of  C_n by DIDD}
Let $W$ be an integral kernel and $n\in \mathbb{N}$.
Then $i_{W,n}$ is measurable and
$$\left\langle \left\{i^{-1}_{W,n}(A):A\in \mathcal{B}(\mathbb{M}_n)\right\}\right\rangle=\mathcal{C}^W_n,$$
i.e., the minimum relatively complete sub-$\sigma$-algebra of $\mathcal{B}$ that makes the map $i_{W,n}$ measurable is $\mathcal{C}_n^W$.
\end{proposition}
\begin{proof}
It is clear that the claim holds for $n=0$ because $\mathcal{C}^W_0=\langle\{\emptyset,X\}\rangle=\left\langle\left\{i^{-1}_{W,0}(\emptyset),i^{-1}_{W,0}(\{\star\}) \right\} \right\rangle$.
Suppose that the claim holds for $n\in \mathbb{N}$.
It follows from \cite[Theorem~17.24]{kechris1995classical} together with the definition of $\mathbb{M}_{n+1}$ that $\mathcal{B}(\mathbb{M}_{n+1})$ is generated by $\{p^{-1}_{n,n+1}(A):A\in \mathcal{B}(\mathbb{M}_n)\}$ and the maps
$$\mathbb{M}_{n+1}\ni\kappa\mapsto \int_{\mathbb{M}_n} f \ d\kappa(n+1)\in \mathbb{R},$$
where $f:\mathbb{M}_n\to \mathbb{R}$ is a bounded Borel function.

Let $A\in \mathcal{B}(\mathbb{M}_n)$.
Then we have
$$i^{-1}_{W,n+1}(p^{-1}_{n,n+1}(A))=i^{-1}_{W,n}(A)\in \mathcal{C}^W_n\subseteq \mathcal{C}^W_{n+1}$$
by the inductive hypothesis.
Let $f:\mathbb{M}_n\to \mathbb{R}$ be a bounded Borel function.
Then the map
$$X\ni x\mapsto \int_{\mathbb{M}_n}f \ d\left(i_{W,n+1}(x)(n+1)\right)=\int_{X} W(x,y)\left(f\circ i_{W,n}\right)(y) \ d\mu(y)$$
is $\mathcal{C}^W_{n+1}$ measurable by the definition of $\mathcal{C}^W_{n+1}$ together with the inductive hypothesis and  Claim~\ref{cl:3-Basic 1}.
This shows that $i_{W,n+1}$ is measurable and $\mathcal{D}_{n+1}\subseteq \mathcal{C}^W_{n+1}$, where we denote as $\mathcal{D}_{n+1}$ the minimum relatively complete sub-$\sigma$-algebra that makes $i_{W,n+1}$ measurable.

It remains to show that $\mathcal{C}^W_{n+1}=\mathcal{D}_{n+1}$.
For $A\in \mathcal{C}^W_n$ we find $B\in \mathcal{B}(\mathbb{M}_n)$ such that $\mu\left(A\triangle i^{-1}_{W,n}(B)\right)=0$ by the inductive hypothesis.
Then we have that the function
$$X\ni x\mapsto i_{W,n+1}(x)(n+1)(B)=\int_{X}W(x,y)\left({\bf 1}_A\right)(y) \ d\mu(y)=T_W({\bf 1}_A)(x)$$
is $\mathcal{D}_{n+1}$ measurable.
An easy argument shows that $\mathcal{C}^W_{n+1}$ is the minimum relatively complete sub-$\sigma$-algebra that makes $\left\{T_W({\bf 1}_A)\right\}_{A\in \mathcal{C}^W_n}$ measurable.
Consequently $\mathcal{D}_{n+1}=\mathcal{C}^W_{n+1}$ and the proof is finished.
\end{proof}

\begin{corollary}\label{cor:Description of C by DIDD}
Let $W$ be an integral kernel.
Then $i_W$ is measurable and
$$\left\langle \left\{i^{-1}_{W}(A):A\in \mathcal{B}(\mathbb{M})\right\}\right\rangle=\mathcal{C}(W),$$
i.e., the minimum relatively complete sub-$\sigma$-algebra of $\mathcal{B}$ that makes the map $i_{W}$ measurable is $\mathcal{C}(W)$.
\end{corollary}
\begin{proof}
It is a standard fact that $\mathcal{B}(\mathbb{M})$ is generated by
$$\bigcup_{n\in \mathbb{N}}\left\{p^{-1}_{n,\infty}(A):A\in \mathcal{B}(\mathbb{M}_n)\right\}$$
as a $\sigma$-algebra (see \cite[Section~10]{kechris1995classical}).
The rest is an easy consequence of the definition of $\mathcal{C}(W)$ together with Proposition~\ref{pr:Description of  C_n by DIDD}
\end{proof}

It remains to show that $\nu_W$ is a DIDM.
By the definition, we have $\nu_W\in \mathscr{P}(\mathbb{M})$.

\begin{proposition}\label{pr:nu_W is DIDD}
Let $W$ be an integral kernel.
Then $\nu_W$ is a DIDM and $i_W(x)\in \mathbb{P}$ for every $x\in X$.
\end{proposition}
\begin{proof}
First we show that $i_W(x)\in \mathbb{P}$ for every $x\in X$.
This immediately implies that $\nu_W(\mathbb{P})=1$.
Let $A\in \mathcal{B}(\mathbb{M}_{n})$.
Then we have
\begin{equation*}
\begin{split}
i_{W}(x)(n+1)(A)= & \ i_{W,n+1}(x)(n+1)(A)=\int_{i^{-1}_{W,n}(A)} W(x,y) \ d\mu(y) \\
= & \ \int_{i^{-1}_{W,n+1}(p^{-1}_{n,n+1}(A))} W(x,y) \ d\mu(y)=i_{W,n+2}(x)(n+2)(p^{-1}_{n,n+1}(A)) \\
= & \ i_{W}(x)(n+2)(p^{-1}_{n,n+1}(A))=(p_{n,n+1})_*\left(i_W(x)(n+2)\right)(A)
\end{split}
\end{equation*}
by the definition of $i_W$.
This shows that $i_W(x)\in \mathbb{P}$ for every $x\in X$.

Let $x\in X$ and write $\mu_x=\mu_{i_W(x)}$.
It follows from Corollary~\ref{cor:Description of C by DIDD} and Corollary~\ref{cor:Quotient} that there is a function $g_x:\mathbb{M}\to [0,1]$ such that
$$\mathbb{E}(W(x,{-})|\mathcal{C}(W))=g_x\circ i_W$$
holds $\mu$-almost everywhere.
We show that $g_x$ is the desired Radon--Nikodym derivative $\frac{d\mu_x}{d\nu_W}$.
To this end, let $A\in \bigcup_{n\in \mathbb{N}} \mathcal{B}(\mathbb{M}_n)$.
Then we have
\begin{equation*}
\begin{split}
\mu_x(p^{-1}_{n,\infty}(A))= & \ i_W(x)(n+1)(A)=\int_{i^{-1}_{W,n}(A)} W(x,{-}) \ d\mu \\
= & \ \int_{i^{-1}_{W,n}(A)} \mathbb{E}\left(W(x,{-})|\mathcal{C}^W_n\right) \ d\mu= \int_{i^{-1}_{W,n}(A)} \mathbb{E}(W(x,{-})|\mathcal{C}(W)) \ d\mu \\
= & \ \int_{i^{-1}_{W,n}(A)} g_x\circ i_W \ d\mu=\int_{i^{-1}_W(p^{-1}_{n,\infty}(A))} g_x\circ i_W \ d\mu=\int_{p^{-1}_{n,\infty}(A)} g_x \ d\nu_W,
\end{split}
\end{equation*}
where the third equality follows from $i^{-1}_{W,n}(A)\in \mathcal{C}^W_n$ by Proposition~\ref{pr:Description of  C_n by DIDD} and the sixth equality by the fact that $x\in i^{-1}_{W,n}(A)$ if and only if $x\in i^{-1}_W(p^{-1}_{n,\infty}(A))$ by the definition of $i_W$.
The rest follows from the fact that $\mu_x$ and $\nu_W$ are well defined and 
$$\bigcup_{n\in \mathbb{N}}\left\{p^{-1}_{n,\infty}(A):A\in \mathcal{B}(\mathbb{M}_n)\right\}$$
generates $\mathcal{B}(\mathbb{M})$.
\end{proof}

\subsection{From DIDM to Integral Kernels}

We start with a DIDM $\nu$ and define an integral kernel ${\bf U}[\nu]$.
Then we show what is the connection between $W$ and ${\bf U}[\nu_W]$.
Recall that by the definition, $\nu$ is concentrated on $\mathbb{P}$ and the map $\alpha\mapsto \mu_\alpha$ is continuous by Claim~\ref{cl:uniformKolmogorov}.
This is enough to get the following.

\begin{claim}\label{cl:Kernel from DIDM}
Let $\nu$ be a DIDM.
Then there is ${\bf U}[\nu]\in L^\infty(\mathbb{M}\times \mathbb{M},\nu\times \nu)$ such that $\|{\bf U}[\nu]\|_\infty\le 1$ and 
$${\bf U}[\nu](\alpha,{-})=\frac{d\mu_{\alpha}}{d\nu}$$
for $\nu$-almost every $\alpha\in \mathbb{M}$.
\end{claim}
\begin{proof}
Let $A\in \mathcal{B}(\mathbb{M}\times \mathbb{M})$ and put $A_\alpha=\{\beta\in \mathbb{M}:(\alpha,\beta)\in A\}$.
Then the assignment
$$\mathbb{M}\ni \alpha\mapsto \mu_\alpha(A_\alpha)\in [0,1]$$
is defined $\nu$-almost everywhere and it is an easy consequence of Claim~\ref{cl:uniformKolmogorov} that it is measurable.
This allows to compute
$$\Phi(A)=\int_{\mathbb{M}} \mu_\alpha(A_\alpha) \ d\nu.$$
It is straightforward to check that $\Phi$ is a Borel probability measure on $\mathbb{M}\times \mathbb{M}$ that is absolutely continuous with respect to $(\nu\times \nu)$.
Let ${\bf U}[\nu]$ be the corresponding Radon--Nikodym derivative.
We leave as an exercise to show that ${\bf U}[\nu](\alpha,{-})=\frac{d\mu_\alpha}{d\nu}$ for $\nu$-almost every $\alpha\in \mathbb{M}$. 
\end{proof}

\begin{theorem}\label{th:DIDD and kernels}
Let $W$ be an integral kernel on $X$.
Then 
$$W_{\mathcal{C}(W)}(x,y)={\bf U}[\nu_W](i_W(x),i_W(y))$$
for $(\mu\times \mu)$-almost every $(x,y)\in X\times X$.
\end{theorem}
\begin{proof}
Recall that by Proposition~\ref{pr:nu_W is DIDD}, we have that ${\bf U}[\nu_W]$ is well defined because $\nu_W$ is a DIDM and $i_W(x)\in \mathbb{P}$ for every $x\in X$.
Consequently, ${\bf U}[\nu_W](i_W(x),{-})=\frac{d\mu_{i_W(x)}}{d\nu_W}$ for $\mu$-almost every $x\in X$ by Claim~\ref{cl:Kernel from DIDM}.

Define an integral kernel $U$ on $X$ as
$$U(x,y)={\bf U}[\nu_W](i_W(x),i_W(y)).$$
It is clearly enough to show that $T_{W_{\mathcal{C}(W)}}=T_U$.
By the definition of $W_{\mathcal{C}(W)}$ and Corollary~\ref{cor:Description of C by DIDD}, we have that $W_{\mathcal{C}(W)}$ and $U$ are $\left(\mathcal{C}(W)\times \mathcal{C}(W)\right)$-measurable.
This implies $T_{W_{\mathcal{C}(W)}}(f)=T_U(f)=0$ whenever $f\in L^2(X,\mathcal{C}(W),\mu)^\bot$.
It is therefore enough to show that $T_{W_{\mathcal{C}(W)}}({\bf 1}_A)=T_{U}({\bf 1}_A)$ for every $A\in \bigcup_{n\in \mathbb{N}}\mathcal{C}^W_n$.

To this end, pick such an $A\in \mathcal{C}^W_n$ for some $n\in \mathbb{N}$.
By Proposition~\ref{pr:Description of  C_n by DIDD}, we may assume (up to a $\mu$-null set) that there is $B\in \mathcal{B}(\mathbb{M}_n)$ such that $A=i^{-1}_{W,n}(B)$.
Recall that it follows from the construction of $i_W$ that $i^{-1}_W(p^{-1}_{n,\infty}(B))=A$.
Then we have
\begin{equation*}
\begin{split}
T_{W_{\mathcal{C}(W)}}({\bf 1}_A)(x)= & \ \int_{A} W(x,{-}) \ d\mu= \int_{i^{-1}_{W,n}(B)} W(x,{-}) \ d\mu \\
= & \ i_W(x)(n+1)(B)= \mu_{i_W(x)}(p^{-1}_{n,\infty}(B)) \\
= & \ \int_{p^{-1}_{n,\infty}(B)} \frac{d\mu_{i_W(x)}}{d\nu_W} \ d\nu_W= \int_{p^{-1}_{n,\infty}(B)} {\bf U}[\nu_W](i_W(x),{-}) \ d\nu_W\\
= & \ \int_{A} U(x,{-}) \ d\mu=T_U({\bf 1}_A)
\end{split}
\end{equation*}
by the definition of $i_W$, $\mu_\alpha$ and ${\bf U}[\nu]$ for $\mu$-almost every $x\in X$.
\end{proof}

\begin{corollary}\label{cor:2 implies 3}
Let $W$ be a graphon.
Then $W/\mathcal{C}(W)$ is isomorphic to ${\bf U}[\nu_W]$.
In particular, ${\bf U}[\nu_W]$ is a graphon.
\end{corollary}
\begin{proof}
By Theorem~\ref{th:Quotients and Markov} and Corollary~\ref{cor:Quotient}, the maps $q_{\mathcal{C}(W)}$ and $i_W$ induce Markov injections $I_{\mathcal{C}(W)}:L^2(X/{\mathcal{C}(W)},\mu/{\mathcal{C}(W)})\to L^2(X,\mu)$ and $I:L^2(\mathbb{M},\nu_W)\to L^2(X,\mu)$ that are isometries onto $L^2(X,{\mathcal{C}(W)},\mu)$.
It follows that
$$(I_{\mathcal{C}(W)})^*\circ I=I^{-1}_{{\mathcal{C}(W)}}\circ I:L^2(\mathbb{M},\nu_W)\to L^2(X/{\mathcal{C}(W)},\mu/{\mathcal{C}(W)})$$
is a Markov isomorphism.
By Theorem~\ref{th:Markov Isomorphism}, we find a measurable measure preserving almost bijection $j_W:X/{\mathcal{C}(W)}\to \mathbb{M}$ such that $i_W=j_W\circ q_{\mathcal{C}(W)}$.
Now it follows easily that $(W/\mathcal{C}(W))(x,y)={\bf U}[\nu_W](j_W(x),j_W(y))$ for $\left((\mu/{\mathcal{C}(W)})\times (\mu/{\mathcal{C}(W)})\right)$-almost every $(x,y)\in (X/{\mathcal{C}(W)})\times (X/{\mathcal{C}(W)})$ by the definition of $W/\mathcal{C}(W)$ and Theorem~\ref{th:DIDD and kernels}.
\end{proof}

\section{Tree functions}\label{sec:Tree functions}

This section is the most technical part of the paper.
We show two things.
First, if $W$ is a graphon and $\mathcal{C}\in \Theta_\mu$ is $W$-invariant, then
$$t(T,W)=t(T,W_\mathcal{C})$$
for every finite tree $T$.
Second, there is a collection $\mathcal{T}\subseteq C(\mathbb{M},\mathbb{R})$ that satisfies assumption of Corollary~\ref{cor:SepMeasures}, i.e., $\mathcal{T}$ separates measures, such that for every $f\in \mathcal{T}$ there is a finite tree $T$ such that 
$$t(T,W)=\int_{\mathbb{M}} f \ d\nu_W$$
for every graphon $W$.

Since we work with arbitrary integral kernels, not necessarily graphons, we state all the results in terms of rooted trees rather than trees.
Recall that for a Borel probability measure $\mu$ on $X$ we denote as $\mu^{\oplus k}$ the Borel probability measure on $X^k$ that is the product of $k$-many copies of $\mu$.

\subsection{Tree Functions and Invariant Subspaces}

A \emph{finite rooted tree} $\mathfrak{T}$ is a pair $(T,v)$, where $T=(V(T),E(T))$ is a finite tree and $v$ is a distinguished vertex of $T$.
The \emph{height, $h(\mathfrak{T})$, of $\mathfrak{T}$} is the maximum number of edges in a path that starts at $v$.
We denote as $c(\mathfrak{T})$ the degree of $v$ in $T$.
Every finite rooted tree $\mathfrak{T}$ of non-zero height can be decomposed into subtrees that are rooted at the neighbors of $v$.
Namely, there is a sequence $\{\mathfrak{T}_i\}_{i\in [c(\mathfrak{T})]}$ of finite rooted trees such that $V(T)=\{v\}\cup\bigcup_{i\in [c(\mathfrak{T})]}V(T_i)$ and $E(T)=\bigcup_{i\in [c(\mathfrak{T})]} \{v,v_i\}\cup E(T_i)$, where $\mathfrak{T}_i=(T_i,v_i)$.
We call $\{\mathfrak{T}_i\}_{i\in c([\mathfrak{T}])}$ the \emph{corresponding decomposition of $\mathfrak{T}$}.
Note that if $h(\mathfrak{T})>0$, then $h(\mathfrak{T}_i)<h(\mathfrak{T})$ for every $i\in [c(\mathfrak{T})]$ and there is $i\in [c(\mathfrak{T})]$ such that $h(\mathfrak{T}_i)+1=h(\mathfrak{T})$.

\begin{definition}
Let $W$ be an integral kernel and $\mathfrak{T}$ be a finite rooted tree.
We define inductively function $f^W_{\mathfrak{T}}:X\to [0,1]$ as follows.
If $h(\mathfrak{T})=0$, then put $f^W_{\mathfrak{T}}=1$.
Suppose that $h(\mathfrak{T})>0$ and define
$$f^W_{\mathfrak{T}}(x)=\int_{X^{[c(\mathfrak{T})]}} \prod_{i\in [c(\mathfrak{T})]} f^W_{\mathfrak{T}_i}(y(i)) W(x,y(i)) \ d\mu^{\oplus c(\mathfrak{T})}(y),$$
where $\{\mathfrak{T}_i\}_{i\in [c(\mathfrak{T})]}$ is the corresponding decomposition of $\mathfrak{T}$.
\end{definition}

\begin{proposition}\label{pr:tree functions and invariant algebras}
Let $W$ be an integral kernel on $X$, $\mathfrak{T}$ be a finite rooted tree and $\mathcal{C}\in \Theta_\mu$ be $W$-invariant.
Then $f^W_{\mathfrak{T}}$ is $\mathcal{C}^W_{h(\mathfrak{T})}$-measurable and $f^{W_\mathcal{C}}_{\mathfrak{T}}(x)=f^{W}_{\mathfrak{T}}(x)$ for $\mu$-almost every $x\in X$.
\end{proposition}
\begin{proof}
We prove both statements simultaneously by induction.
If $h(\mathfrak{T})=0$, then the claim clearly holds.
Suppose that $h(\mathfrak{T})=n+1$ and that the claim holds for all finite rooted trees of height at most $n$.
Let $\{\mathfrak{T}_i\}_{i\in c(\mathfrak{T})}$ be the corresponding decomposition of $\mathfrak{T}$.
We have
\begin{equation*}
\begin{split}
f^W_{\mathfrak{T}}(x)= & \ \int_{X^{[c(\mathfrak{T})]}} \prod_{i\in [c(\mathfrak{T})]} f^W_{\mathfrak{T}_i}(y(i)) W(x,y(i)) \ d\mu^{\oplus c(\mathfrak{T})}(y) \\
= & \ \prod_{i\in [c(\mathfrak{T})]}\left(\int_{X} f^W_{\mathfrak{T}_i}(y) W(x,y) \ d\mu(y) \right)=\prod_{i\in [c(\mathfrak{T})]}\left(\int_{X} f^{W_\mathcal{C}}_{\mathfrak{T}_i}(y) W(x,y) \ d\mu(y) \right) \\
= & \ \prod_{i\in [c(\mathfrak{T})]}\left(\int_{X} f^{W_\mathcal{C}}_{\mathfrak{T}_i}(y) \mathbb{E}(W(x,{-})|\mathcal{C})(y) \ d\mu(y) \right) \\
= & \ \prod_{i\in [c(\mathfrak{T})]}\left(\int_{X} f^{W_\mathcal{C}}_{\mathfrak{T}_i}(y) W_\mathcal{C}(x,y) \ d\mu(y) \right)=f^{W_{\mathcal{C}}}_{\mathfrak{T}}(x) \\
\end{split}
\end{equation*}
for $\mu$-almost every $x\in X$, where the second equality is Fubini's Theorem, the third is by inductive hypothesis, the fourth follows from Theorem~\ref{th:ConditionalExp}~(2) together with $\mathcal{C}^W_n\subseteq \mathcal{C}(W)\subseteq \mathcal{C}$ and the fifth follows from the fact that $\mathbb{E}(W(x,{-})|\mathcal{C})=W_{\mathcal{C}}(x,{-})$ for $\mu$-almost every $x\in X$.
Note that by the definition of $\mathcal{C}^W_{n+1}$, we have that $f^W_{\mathfrak{T}}$ is $\mathcal{C}^W_{n+1}$-measurable by the second equality and that finishes the proof.
\end{proof}

\begin{proposition}\label{pr:tree functions and graphon}
Let $W$ be a graphon on $X$, $\mathfrak{T}=(T,v)$ be a finite rooted tree and $\mathcal{C}\in \Theta_\mu$ be $W$-invariant.
Then
$$t(T,W)=\int_{X} f^W_{\mathfrak{T}}(x) \ d\mu(x).$$
In particular, $t(T,W)=t(T,W_\mathcal{C})=t(T,{\bf U}[\nu_W])$ for every finite tree $T$.
\end{proposition}
\begin{proof}
If $h(\mathfrak{T})=0$, then the claim holds.
Suppose that $h(\mathfrak{T})=n+1$ and $\{\mathfrak{T}_i\}_{i\in c([\mathfrak{T}])}$ is the corresponding decomposition of $\mathfrak{T}$, where $\mathfrak{T}_i=(T_i,v_i)$.
It is easy to see by induction on $h(\mathfrak{T})$ together with Fubini's Theorem that for fixed $x\in [0,1]$ we have
\begin{equation*}
\int_{X} W(x,y) f^{W}_{\mathfrak{T}_i}(y) \ d\mu(y) =\int_{X^{V(T_i)}} W(x,y(v_i)) \prod_{\{w,u\}\in E(T_i)} W(y(w),y(u)) \ d\mu^{\oplus |V(T_i)|}(y)
\end{equation*}
and that gives immediately
\begin{equation*}
\begin{split}
t(T,W)= & \ \int_{X^{V(T)}} \prod_{\{w,u\}\in E(T)} W(y(w),y(u)) \ d\mu^{\oplus |V(T)|}(y) \\
= & \ \int_{X} \prod_{i\in [c(\mathfrak{T})]}\left(\int_{X^{V(T_i)}} W(x,y(v_i)) \prod_{\{w,u\}\in E(T_i)} W(y(w),y(u)) \ d\mu^{\oplus |V(T_i)|}(y)\right) \ d\mu(x) \\
= & \ \int_{X} \prod_{i\in [c(\mathfrak{T})]}\left(\int_{X} W(x,y) f^{W}_{\mathfrak{T}_i}(y) \ d\mu(y)\right) \ d\mu(x) \\
= & \ \int_{X} f^W_{\mathfrak{T}}(x) \ d\mu(x)
\end{split}
\end{equation*}
as desired.
Note that the assumption that $W$ is symmetric is implicitly used in the second equality.

It follows from Proposition~\ref{pr:tree functions and invariant algebras} that $t(T,W)=t(T,W_\mathcal{C})$.
In particular, we have $t(T,W)=t(T,W_{\mathcal{C}(W)})$ and $t(T,W_{\mathcal{C}(W)})=t(T,{\bf U}[\nu_W])$ by Proposition~\ref{pr:Quotient and Invariant algebra} together with Corollary~\ref{cor:2 implies 3}.
\end{proof}

\subsection{Collection $\mathcal{T}$}\label{subsec:def of T}

In this section we work exclusively with the space $\mathbb{M}$.
We define a collection $\mathcal{T}\subseteq C(\mathbb{M},\mathbb{R})$ that is closed under multiplication and contains ${\bf 1}_{\mathbb{M}}$.
The construction proceeds recursively on $n\in \mathbb{N}$, where in step $n\in \mathbb{N}$ we construct $\mathcal{T}_n\subseteq C(\mathbb{M},\mathbb{R})$ that factors through $\mathbb{M}_n$, i.e., for every $f\in \mathcal{T}_n$ there is $f'\in C(\mathbb{M}_n,\mathbb{R})$ such that $f=f'\circ p_{n,\infty}$, and is uniformly dense in $C(\mathbb{M}_n,\mathbb{R})\circ p_{n,\infty}$.

The set $\mathcal{T}_{n+1}$ is constructed from $\mathcal{T}_n$ using two operations.
Informally, these operations correspond to the following constructions on finite trees, the correspondence is made precise in the proof of Proposition~\ref{pr:T and tree functions}.
{\bf (I)} Given a rooted tree we add an extra vertex that is the new root and its only neighbor is the old root.
{\bf (II)} Given a sequence of rooted trees $\{\mathfrak{T}^j\}_{j\in [k]}$ we define a rooted tree $\mathfrak{T}$ as a disjoint union of $\{\mathfrak{T}^j\}_{j\in [k]}$ and glue the roots to a single vertex, the new root.

\begin{definition}
Let $n,k\in \mathbb{N}$ and $f,f_1,\dots, f_k\in C(\mathbb{M},\mathbb{R})$ be such that $f$ factors through $\mathbb{M}_n$.
Then define for every $\alpha\in \mathbb{M}$
\begin{itemize}
	\item [{\bf(I)}] $F(f,n)(\alpha)=\int_{\mathbb{M}_n} f' \ d\alpha(n+1)$, where $f'\in C(\mathbb{M}_n,\mathbb{R})$ and $f=f'\circ p_{n,\infty}$,
	\item [{\bf(II)}] $G(f_1,\dots, f_k)(\alpha)=\prod_{j\in [k]}f_j(\alpha)$
\end{itemize}
\end{definition}

It is easy to see by the definition of $\mathbb{M}$ that $F(f,n)$ and $G(f_1,\dots, f_k)$ are elements of $C(\mathbb{M},\mathbb{R})$ and that $F(f,n)$ factors through $\mathbb{M}_{n+1}$.

We put $\mathcal{T}_0=\{{\bf 1}_{\mathbb{M}}\}$.
Suppose that $\mathcal{T}_n$ is defined.
Then let
$$\mathcal{T}_{n+1}=\left\{G(f_1,\dots, f_k):\forall i\in [k] \ \exists g_i\in \mathcal{T}_n \ (g_i=f_i \ \vee F(g_i,n)=f_i)\right\},$$
i.e., first apply {\bf (I)} on $\mathcal{T}_n$ and then {\bf (II)} on all new and old functions.
Finally, we put $\mathcal{T}=\bigcup_{n\in \mathbb{N}}\mathcal{T}_n$.

\begin{proposition}\label{pr:separate points}
The collection $\mathcal{T}$ is closed under multiplication, contains ${\bf 1}_{\mathbb{M}}$ and separates points of $\mathbb{M}$.
\end{proposition}
\begin{proof}
We only need to show that $\mathcal{T}$ separates points.
We show by induction on $n\in \mathbb{N}$ that $\mathcal{T}_n$ separates $\alpha, \beta\in \mathbb{M}$ whenever there is $i\in [n]$ such that $\alpha(i)\not=\beta(i)$.
This clearly suffices to prove the claim.
Note that each $\mathcal{T}_n$ is closed under multiplication and contain ${\bf 1}_{\mathbb{M}}$ by {\bf (II)}.

If $n=0$ there is nothing to prove.
Suppose that the claim holds for $n\in \mathbb{N}$.
Let $\alpha\not=\beta\in \mathbb{M}$ be such that $\alpha(i)\not=\beta(i)$ for some $i\in [n+1]$.
Either there is $f\in \mathcal{T}_n$ such that $f(\alpha)\not=f(\beta)$ or $i=n+1$ by the inductive assumption.
Let $\mathcal{T}'_n=\{f'\in C(\mathbb{M}_n,\mathbb{R}):\exists f\in \mathcal{T}_n \ f=f'\circ p_{n,\infty}\}$.
It follows by the inductive assumption that $\mathcal{T}'_n$ is closed under multiplication, contain ${\bf 1}_{\mathbb{M}_n}$ and separates points of $\mathbb{M}_n$.
By Corollary~\ref{cor:SepMeasures}, there is $f'\in \mathcal{T}'_n$ such that 
$$\int_{\mathbb{M}_n} f' \ d\alpha(n+1)\not=\int_{\mathbb{M}_n} f' \ d\beta(n+1).$$
By {\bf (I)}, we have $F(f,n)(\alpha)\not=F(f,n)(\beta)$, where $f\in \mathcal{T}_n$ is such that $f=f'\circ p_{n,\infty}$.
Since $F(f,n)\in \mathcal{T}_{n+1}$ the proof is finished.
\end{proof}

\begin{proposition}\label{pr:T and tree functions}
Let $f\in \mathcal{T}$.
Then there is a finite rooted tree $\mathfrak{T}$ such that for every DIDM $\nu$ we have
$$f(\alpha)=f^{{\bf U}[\nu]}_{\mathfrak{T}}(\alpha)$$
for $\nu$-almost every $\alpha\in \mathbb{M}$.
\end{proposition}
\begin{proof}
We prove the claim by induction on $n\in \mathbb{N}$.
It is easy to see that if $f={\bf 1}_\mathbb{M}$, then $\mathfrak{T}$ that satisfies $h(\mathfrak{T})=0$ works, i.e., the claim holds for $\mathcal{T}_0$.

Suppose that the claim holds for $\mathcal{T}_n$, where $n\in \mathbb{N}$.
Let $f=F(g,n)$ for some $g\in \mathcal{T}_n$.
Fix a finite rooted tree  $\mathfrak{S}=(S,w)$ that corresponds to $g$ and $g'\in C(\mathbb{M}_n,\mathbb{R})$ such that $g=g'\circ p_{n,\infty}$.
Define a finite rooted tree $\mathfrak{T}$ such that $c(\mathfrak{T})=1$ and $\{\mathfrak{S}\}$ is the corresponding decomposition of $\mathfrak{T}$, i.e., we add an extra vertex that is the new root and its only neighbor is the old root.
Given a DIDM $\nu$ we have
\begin{equation*}
\begin{split}
f^{{\bf U}[\nu]}_{\mathfrak{T}}(\alpha)= & \ \int_{\mathbb{M}} f^{{\bf U}[\nu]}_{\mathfrak{S}}(\beta){\bf U}[\nu](\alpha,\beta) \ d\nu(\beta) =\int_{\mathbb{M}} g(\beta){\bf U}[\nu](\alpha,\beta) \ d\nu(\beta) \\
= & \ \int_{\mathbb{M}} g \ d\mu_\alpha=\int_{\mathbb{M}} g'\circ p_{n,\infty} \ d\mu_\alpha=\int_{\mathbb{M}_n} g' \ d(p_{n,\infty})_*\mu_\alpha \\
= & \ \int_{\mathbb{M}_n} g' \ d\alpha(n+1)=F(g,n)(\alpha)=f(\alpha)
\end{split}
\end{equation*}
for $\nu$-almost every $\alpha\in \mathbb{M}$.

Let $f\in \mathcal{T}_{n+1}$.
By the definition, we have $f=G(f_1,\dots, f_k)$ for some $f_i$ such that either $f_i\in \mathcal{T}_n$ or $f_i=F(g_i,n)$ for some $g_i\in \mathcal{T}_n$.
In both cases, either by inductive assumption or by previous paragraph, we find a finite rooted tree $\mathfrak{T}^i$ that satisfy the claim for $f_i$ for every $i\in [k]$.
Let $\{\mathfrak{T}^i_{j}\}_{j\in [c(\mathfrak{T}^i)]}$ be the corresponding decomposition of $\mathfrak{T}_i$, where $\mathfrak{T}^i_j=(T^i_j,v^i_j)$ for every $i\in [k]$.
Put $I=\{(i,j):i\in [k] \ j\in [c(\mathfrak{T}^i)]\}$ and define $\mathfrak{T}=(T,v)$ as
\begin{equation*}
V(T)=\{v\}\cup\bigcup_{(i,j)\in I}V(T^i_{j}) \ \operatorname{and} \ E(T)=\bigcup_{(i,j)\in I}\{v,v^i_j\}\cup E(T^i_j).
\end{equation*}
Note that $\{\mathfrak{T}^i_{j}\}_{(i,j)\in I}$ is the corresponding decomposition of $\mathfrak{T}$.
Given a DIDM $\nu$ we have
\begin{equation*}
\begin{split}
f^{{\bf U}[\nu]}_{\mathfrak{T}}(\alpha)= & \ \int_{\mathbb{M}^I} \prod_{(i,j)\in I} f^{{\bf U}[\nu]}_{\mathfrak{T}^i_j}(\beta(i,j)){\bf U}[\nu](\alpha,\beta(i,j)) \ d\nu^{\oplus |I|}(\beta)\\
= & \ \prod_{i\in [k]} \int_{\mathbb{M}^{[c(\mathfrak{T}^i)}]} \prod_{j\in c([\mathfrak{T}^i])} f^{{\bf U}[\nu]}_{\mathfrak{T}^i_j}(\beta(j)) {\bf U}[\nu](\alpha,\beta(j)) \ d\nu^{\oplus c(\mathfrak{T}^i)}(\beta) \\
= & \ \prod_{i\in [k]} f^{{\bf U}[\nu]}_{\mathfrak{T}^i}(\alpha)=\prod_{i\in [k]}f_i(\alpha)=f(\alpha)
\end{split}
\end{equation*}
for $\nu$-almost every $\alpha\in \mathbb{M}$ and that finishes the proof.
\end{proof}

\begin{corollary}\label{cor:1 implies 2}
The map $W\mapsto \nu_W$ is continuous when $\mathcal{W}_0$ is endowed with the cut-distance and $\mathscr{P}(\mathbb{M})$ with the weak* topology.
Moreover, if $U$ and $W$ are graphons such that $\nu_W\not=\nu_U$, then there is a finite tree $T$ such that $t(T,W)\not=t(T,U)$.
\end{corollary}
\begin{proof}
It follows from Theorem~\ref{th:StoneWeierstrass} together with Proposition~\ref{pr:separate points} that $\mathcal{T}$ is uniformly dense in $C(\mathbb{M},\mathbb{R})$.
It follows that the weak* topology on $\mathscr{P}(\mathbb{M})$ is generated by functionals that correspond to elements of $\mathcal{T}$.
Let $W_n\xrightarrow{\delta_\Box} W$ and $f\in \mathcal{T}$.
Fix a finite (rooted) tree $T$ that corresponds to $f$ as in Proposition~\ref{pr:T and tree functions}.
By Propositions~\ref{pr:tree functions and graphon},~\ref{pr:T and tree functions}, we have
$$\int_{\mathbb{M}} f \ d\nu_{W_n}=t(T,{\bf U}[\nu_{W_n}])=t(T,W_n)\to t(T,W)=t(T,{\bf U}[\nu_W])=\int_{\mathbb{M}} f \ d\nu_W.$$
That shows that the assignment is continuous.

Suppose that $\nu_W\not=\nu_U$.
By Corollary~\ref{cor:SepMeasures} together with Proposition~\ref{pr:separate points}, we find $f\in \mathcal{T}$ such that 
$$\int_{\mathbb{M}} f\ \nu_W\not= \int_{\mathbb{M}} f \ d\nu_U.$$
A finite (rooted) tree $T$ that corresponds to $f$ as in Proposition~\ref{pr:T and tree functions} satisfies
$$t(T,W)\not= t(T,U)$$
by Proposition~\ref{pr:tree functions and graphon}.
\end{proof}

\section{Proof of Theorem~\ref{th:main}}\label{sec:Proof}

We recall the statement.

\begin{theorem}
Let $W$ and $U$ be graphons.
Then the following are equivalent:
\begin{enumerate}
	\item $t(T,W)=t(T,U)$ for every finite tree $T$,
	\item $\nu_W=\nu_U$,
	\item $W/\mathcal{C}(W)$ and $U/\mathcal{C}(U)$ are isomorphic,
	\item there is a Markov operator $S:L^2(X,\mu)\to L^2(X,\mu)$ such that $T_W\circ S=S\circ T_U$,
	\item there is a $W$-invariant sub-$\sigma$-algebra $\mathcal{C}$ and a $U$-invariant sub-$\sigma$-algebra $\mathcal{D}$ such that $W_\mathcal{C}$ and $U_{\mathcal{D}}$ are weakly isomorphic.
\end{enumerate}
\end{theorem}

\begin{proof}[Proof of Theorem~\ref{th:main}]

{\bf (1) $\Rightarrow$ (2)} Follows immediately from Corollary~\ref{cor:1 implies 2}.

{\bf (2) $\Rightarrow$ (3)}.
Follows from Corollary~\ref{cor:2 implies 3} applied twice to both $W$ and $U$.

{\bf (3) $\Rightarrow$ (4)}.
See paragraph after Claim~\ref{cl:Basic 1} for definitions.
We let $Y=X/\mathcal{C}(W)$, $Z=X/\mathcal{C}(U)$, $\mu_Y=\mu/\mathcal{C}(W)$, $\mu_Z=\mu/\mathcal{C}(U)$, $W_Y=W/\mathcal{C}(W)$ and $U_Z=U/\mathcal{C}(U)$.
By (3), there is a measure preserving isomorphism $j:Y\to Z$ such that
$$W_Y(x,y)=U_Z(j(x),j(y))$$
for $\left(\mu_Y\times \mu_Y\right)$-almost every $(x,y)\in Y\times Y$.
The map
$$S_j:L^2(Y,\mu_Y)\to L^2(Z,\mu_Z)$$
defined as $S_j(f)(x)=f(j^{-1}(x))$ is a Markov isomorphism by Theorem~\ref{th:Markov Isomorphism} and it is routine to check that $S_j\circ T_{W_Y}=T_{U_Z}\circ S_j$.

By Proposition~\ref{pr:Quotient and Invariant algebra}~(iii), we have $T_{W_Y}\circ S_{\mathcal{C}(W)}=S_{\mathcal{C}(W)}\circ T_W$ and 
$$I_{\mathcal{C}(U)}\circ T_{U_Z}=\left(T_{U_Z}\circ S_{\mathcal{C}(U)}\right)^*=\left(S_{\mathcal{C}(U)}\circ T_U\right)^*=T_U\circ I_{\mathcal{C}(U)}.$$
We define a Markov operator $S=I_{\mathcal{C}(U)}\circ S_j\circ S_{\mathcal{C}(W)}$.
It is easy to check that
$$S\circ T_W=T_U\circ S$$
and that finishes the proof.

{\bf (4) $\Rightarrow$ (5)}.
Let $S$ be a Markov operator such that $T_W\circ S=S\circ T_U$.
Then $S\circ S^*$ and $S^*\circ S$ are self-adjoint Markov operators by Proposition~\ref{pr:Basic Markov}.
We have
\begin{equation*}
\begin{split}
T_W\circ (S\circ S^*)= & \ S\circ (T_U \circ S^*)= S\circ (S\circ T_U)^* \\
= & \ S\circ (T_W\circ S)^*=(S\circ S^*)\circ T_W
\end{split}
\end{equation*}
and similarly $T_U\circ (S^*\circ S)=(S^*\circ S)\circ T_U$ because $T_W$ and $T_U$ are self-adjoint by Claim~\ref{cl:Self Adjoint=Graphon}.
In particular, we have 
$$T_W\circ \left(\sum_{k\in [n]}(S\circ S^*)^k\right)=\left(\sum_{k\in [n]}(S\circ S^*)^k\right)\circ T_W \ \operatorname{and} \ T_U\circ \left(\sum_{k\in [n]}(S^*\circ S)^k\right)=\left(\sum_{k\in [n]}(S^*\circ S)^k\right)\circ T_U$$
for every $n\in \mathbb{N}$.

Let $P$ be the orthogonal projection onto $\{f\in L^2(X,\mu):(S\circ S^*)(f)=f\}$ and $Q$ be the orthogonal projection onto $\{f\in L^2(X,\mu):(S^*\circ S)(f)=f\}$.
By the Mean Ergodic Theorem, Theorem~\ref{th:Mean Ergodic}, we have
$$\left\|\frac{1}{n}\sum_{k\in [n]}(S\circ S^*)^k(f)-P(f)\right\|_2\to 0 \ \text{ and } \ \left\|\frac{1}{n}\sum_{k\in [n]}(S^*\circ S)^k(f)-Q(f)\right\|_2\to 0$$
for every $f\in L^2(X,\mu)$.
It follows from Proposition~\ref{pr:Basic Markov} that $P$ and $Q$ are Markov projections and by Theorem~\ref{th:MarkovProj} that there are relatively complete sub-$\sigma$-algebras $\mathcal{C}$ and $\mathcal{D}$ such that $P=\mathbb{E}({-}|\mathcal{C})$ and $Q=\mathbb{E}({-}|\mathcal{D})$.

Let $f\in L^2(X,\mu)$.
Then we have
\begin{equation*}
\begin{split}
\left\|(P\circ S)(f)-(S\circ Q)(f) \right\|_2\le & \ \left\|(P\circ S)(f)-\left(\frac{1}{n}\sum_{k\in [n]}(S\circ S^*)^k\circ S\right)(f)\right\|_2 \\
& \ + \left\|\left(\frac{1}{n}\sum_{k\in [n]}(S\circ S^*)^k\circ S\right)(f)-\left(S\circ \frac{1}{n}\sum_{k\in [n]}(S^*\circ S)^k\right)(f)\right\|_2 \\
& \ + \left\|\left(S\circ \frac{1}{n}\sum_{k\in [n]}(S^*\circ S)^k\right)(f)-(S\circ Q)(f)\right\|_2 \\ 
= & \ \left\|P(S(f))-\frac{1}{n}\sum_{k\in [n]}(S\circ S^*)^k(S(f))\right\|_2 \\
& \ + \left\|S\left(\frac{1}{n}\sum_{k\in [n]}(S^*\circ S)^k(f)-Q(f)\right)\right\|_2 \\
\to & \ 0
\end{split}
\end{equation*}
and similarly $S^*\circ P=Q\circ S^*$.

Let $f\in L^2(X,\mathcal{D},\mu)$.
Then we have $Q(f)=\mathbb{E}(f|\mathcal{D})=f$ and $P(S(f))=S(Q(f))=S(f)$ by the previous paragraph.
Moreover,
$$\|S(f)\|^2_2=\langle S(f),S(f)\rangle=\langle (S^*\circ S)(f),f\rangle=\langle f,f \rangle=\|f\|_2$$
by the definition of $Q$.
This shows that $S\upharpoonright L^2(X,\mathcal{D},\mu)$ is an isometric embedding into $L^2(X,\mathcal{C},\mu)$.
A similar argument shows that $S^*\upharpoonright L^2(X,\mathcal{C},\mu)$ is an isometric embedding into $ L^2(X,\mathcal{D},\mu)$.
Since $S^*\circ S$ is identity when restricted to $L^2(X,\mathcal{D},\mu)$ and similarly for $S\circ S^*$ we conclude that $S$ is an isometrical isomorphism between $L^2(X,\mathcal{D},\mu)$ and $L^2(X,\mathcal{C},\mu)$.

Putting this together with properties of quotients, see definitions after Claim~\ref{cl:Basic 1}, we get that 
$$R=S_{\mathcal{C}}\circ S\circ I_\mathcal{D}:L^2(X/\mathcal{D},\mu/\mathcal{D})\to L^2(X/\mathcal{C},\mu/\mathcal{C})$$
is a Markov isomorphism such that 
\begin{equation*}
\begin{split}
R\circ T_{U/\mathcal{D}}= & S_{\mathcal{C}}\circ S\circ I_\mathcal{D}\circ T_{U/\mathcal{D}}=S_{\mathcal{C}}\circ S\circ T_U\circ I_\mathcal{D}\\
= & \ S_{\mathcal{C}}\circ T_W\circ S\circ I_\mathcal{D}=T_{W/\mathcal{C}}\circ S_{\mathcal{C}}\circ S\circ I_\mathcal{D} \\
= & \ T_{W/\mathcal{C}}\circ R.
\end{split}
\end{equation*}
By Theorem~\ref{th:Markov Isomorphism}, there is a measure preserving (almost) bijection $i:X/\mathcal{D}\to X/\mathcal{C}$ such that $R(f)(x)=f(i^{-1}(x))$.
We show that $(U/\mathcal{D})(i^{-1}(x),i^{-1}(y))=(W/\mathcal{C})(x,y)$ for $((\mu/\mathcal{C})\times (\mu/\mathcal{C}))$-almost every $(x,y)\in (X/\mathcal{C})\times (X/\mathcal{C})$.
This implies that $U/\mathcal{D}$ and $W/\mathcal{C}$ are isomorphic and consequently $U_\mathcal{D}$ and $W_\mathcal{C}$ are weakly isomorphic, by Proposition~\ref{pr:Quotient and Invariant algebra}~(i), as desired.

Let $V$ be a graphon on $X/\mathcal{C}$ defined as $V(x,y)=(U/\mathcal{D})(i^{-1}(x),i^{-1}(y))$ and $f,g\in L^2(X/\mathcal{C},\mu/\mathcal{C})$.
We have
\begin{equation*}
\begin{split}
\langle T_{W/\mathcal{C}}(f),g\rangle= & \langle R^{-1}(T_{W/\mathcal{C}}(f)),R^{-1}(g)\rangle=\langle T_{U/\mathcal{D}}(R^{-1}(f)),R^{-1}(g)\rangle\\
= & \ \int_{(X/\mathcal{D})\times (X/\mathcal{D})} f(i(x))(U/\mathcal{D})(x,y)g(i(y)) \ d((\mu/\mathcal{D})\times (\mu/\mathcal{D}))(x,y) \\
= & \ \int_{(X/\mathcal{C})\times (X/\mathcal{C})} f(x)V(x,y)g(y) \ d((\mu/\mathcal{C})\times (\mu/\mathcal{C}))(x,y)=\langle T_V(f),g\rangle.
\end{split}
\end{equation*}
That shows $T_{W/\mathcal{C}}=T_V$, consequently $W/\mathcal{C}=V$ and the proof is finished.

{\bf (5) $\Rightarrow$ (1)}.
It follows from Proposition~\ref{pr:tree functions and graphon} that $t(T,W)=t(T,W_\mathcal{C})$ whenever $T$ is a tree and $\mathcal{C}$ is $W$-invariant, and similarly $t(T,U)=t(T,U_\mathcal{D})$. Since, $W_\mathcal{C}$ and $U_\mathcal{D}$ are weakly isomorphic, we have $t(T,W_\mathcal{C})=t(T,U_\mathcal{D})$ and that finishes the proof.
\end{proof}

\section*{Acknowledgments}

The authors are grateful to Jan Hladk\' y for useful discussions and help, and to anonymous referees for many useful suggestions and comments that helped to improve the presentation of the paper.
The first author also thanks Jan Byd\v zovsk\' y, Jan Hladk\' y and Oleg Pikhurko for help with the current version of the introduction.

\appendix
\section{Standard Borel spaces}\label{App A}

Let $X$ be a set and $\mathcal{B}$ a $\sigma$-algebra of subsets of $X$.
We say that $(X,\mathcal{B})$ is a {\it standard Borel space} if there is a separable completely metrizable topology $\tau$ on $X$ such that $\mathcal{B}$ is equal to the $\sigma$-algebra of Borel subsets generated by $\tau$ (see \cite[Section~12]{kechris1995classical}).
We denote the space of all Borel probability measures on $X$  as $\mathscr{P}(X)$ and the space of all measures of total mass at most $1$ as $\mathscr{M}_{\le 1}(X)$.
Note that the sets $\mathscr{P}(X)$ and $\mathscr{M}_{\le 1}(X)$ endowed with the $\sigma$-algebra generated by the maps
$$A\mapsto \mu(A),$$
where $A\in \mathcal{B}$, are standard Borel spaces (see \cite[Section~17]{kechris1995classical}).

Let $\mu,\nu\in \mathscr{M}_{\le 1}(X)$.
We say that $\nu$ is \emph{absolutely continuous with respect to $\mu$} if $\mu(A)=0$ whenever $\nu(A)=0$.
The classical Radon--Nikodym Theorem \cite[Theorem 6.10]{Rud} states that this occurs if and only if there is a unique $f\in L^1(X,\mu)$ such that
$$\nu(A)=\int_A f \ d\mu$$
for every $A\in \mathcal{B}$.
We call $f$ the {\it Radon--Nikodym derivative of $\nu$ with respect to $\mu$} and denote it as $\frac{d\nu}{d\mu}$.

Let $(X,\mathcal{B})$ and $(Y,\mathcal{C})$ be standard Borel spaces.
Suppose that $\mu\in \mathscr{P}(X)$ and $f:X\to Y$ is a Borel map.
Then we define {\it the push-forward of $\mu$ via $f$}, in symbols $f_*\mu$, as 
$$f_*\mu(A)=\mu(f^{-1}(A))$$
for every $A\in \mathcal{C}$.
It is a standard fact that $f_*\mu\in \mathscr{P}(Y)$, see \cite[Exercise~17.28]{kechris1995classical}.

\section{Compact Spaces}\label{App B}

Let $K$ be a compact metric space.
Write $C(K,\mathbb{R})$ for the vector space of all continuous functions from $K$ to $\mathbb{R}$.
Then $C(K,\mathbb{R})$ with the supremum norm and pointwise multiplication is a real Banach algebra.
We denote the $\sigma$-algebra of Borel sets of $K$ as $\mathcal{B}(K)$.
Then $(K,\mathcal{B}(K))$ is a standard Borel space.

It is a standard fact, see~\cite[Section~17]{kechris1995classical}, that the space of Borel measures of total mass at most $1$, i.e., $\mathscr{M}_{\le 1}(K)$, coincides with the space of all positive real-valued Radon measures of total mass at most $1$.
By the Riesz Representation Theorem~\cite[Theorem 6.19]{Rud}, these are exactly the positive linear functionals with norm at most $1$ in the dual space of $C(K,\mathbb{R})$.
The weak* topology on $\mathscr{M}_{\le 1}(K)$ is then defined as the coarsest topology that makes the maps
$$\int_K f \ d\mu_n\to \int_K f \ d\mu$$
continuous for every $f\in C(K,\mathbb{R})$.
It is a standard fact that $\mathscr{M}_{\le 1}(K)$ endowed with the weak* topology is compact metrizable space, see~\cite[Theorem~17.22]{kechris1995classical}, and that the $\sigma$-algebra of Borel sets generated by the weak* topology on $\mathscr{M}_{\le 1}(K)$ coincides with the standard Borel structure on $\mathscr{M}_{\le 1}(K)$ generated by the maps
$$A\mapsto \mu(A),$$
where $A\in \mathcal{B}(K)$ (see \cite[Section~17]{kechris1995classical}).

\begin{theorem}[Real Stone--Weierstrass]\cite[Theorem~7.32]{Rud3}\label{th:StoneWeierstrass}
Let $K$ be a compact metric space and $\mathcal{A}\subseteq C(K,\mathbb{R})$ be a subalgebra that contains ${\bf 1}_K$ and separates points, i.e., for every $k\not=l\in K$ there is $f\in \mathcal{A}$ such that $f(k)\not=f(l)$.
Then $\mathcal{A}$ is uniformly dense in $C(K,\mathbb{R})$.
\end{theorem}

\begin{corollary}[Separating Measures]\label{cor:SepMeasures}
Let $K$ be a compact metric space and $\mathcal{E}\subseteq C(K,\mathbb{R})$ be closed under multiplication, contain ${\bf 1}_K$, and separate points.
Then for every $\mu\not=\nu\in \mathscr{M}_{\le 1}(K)$ there is $f\in \mathcal{E}$ such that
$$\int_K f \ d\mu\not= \int_K f \ d\nu,$$
i.e., the linear functionals that correspond to elements of $\mathcal{E}$ separate points in $\mathscr{M}_{\le 1}(K)$.
\end{corollary}


\section{Conditional Expectation}\label{App C}

Let $(X,\mathcal{B})$ be a standard Borel space and $\mu\in \mathscr{P}(X)$.
A sub-$\sigma$-algebra $\mathcal{C}$ of $\mathcal{B}$ is {\it relatively complete} if $Z\in \mathcal{C}$ whenever there is $Z_0\in \mathcal{C}$ such that $\mu(Z\triangle Z_0)=0$.
We denote the collection of all relatively complete sub-$\sigma$-algebras as $\Theta_\mu$.

If $\mathcal{C}\in \Theta_\mu$ and $(Y,\mathcal{D})$ is a standard Borel space, then we say that a map $f:X\to Y$ is $\mathcal{C}$-measurable if $f^{-1}(A)\in \mathcal{C}$ for every $A\in \mathcal{D}$.
We denote as $L^2(X,\mathcal{C},\mu)$ the closed linear subspace of $L^2(X,\mu)$ that consists of $\mathcal{C}$-measurable functions.

\begin{theorem}\cite[Section~34]{billingsley}\label{th:ConditionalExp}
Let $(X,\mathcal{B})$ be a standard Borel space, $\mu$ be a Borel probability measure and $\mathcal{C}\in \Theta_\mu$.
Then there is a bounded self-adjoint linear operator
$$\mathbb{E}({-}|\mathcal{C}):L^2(X,\mu)\to L^2(X,\mathcal{C},\mu)$$
that enjoys the following properties:
\begin{enumerate}
	\item $\mathbb{E}({-}|\mathcal{C})$ is the orthogonal projection onto $L^2(X,\mathcal{C},\mu)$,
	\item $\int_X f\mathbb{E}(g|\mathcal{C}) \ d\mu=\int_X \mathbb{E}(f|\mathcal{C})g \ d\mu$ for every $f,g\in L^2(X,\mu)$,
	\item for every $A\in \mathcal{C}$ and $f\in L^2(X,\mu)$ we have 
	$$\int_A f \ d\mu=\int_A \mathbb{E}(f|\mathcal{C}) \ d\mu.$$
\end{enumerate}
\end{theorem}

\section{Markov Operators}\label{App D}

We need the theory of Markov operators for the correspondence between Markov projections and relatively complete sub-$\sigma$-algebras, and for the Mean Ergodic Theorem.
Our main reference is \cite{eisner2015operator}.
We point out that it is more convenient for us to define and work with Markov operators on $L^2$ spaces rather than on $L^1$ spaces (as it is defined in \cite{eisner2015operator}).
However, it follows from \cite[Chapter~13,~Proposition~13.6]{eisner2015operator} that every Markov operator on $L^2$ space has a unique extension to a Markov operator on $L^1$ space and that the restriction of a Markov operator on $L^1$ space to $L^2$ space is a Markov operator.

Let $(X,\mathcal{B})$ and $(Y,\mathcal{D})$ be standard Borel spaces  with Borel probability measures $\mu$ and $\nu$, respectively.
We say that a bounded linear operator $S:L^2(X,\mu)\to L^2(Y,\nu)$ is a {\it Markov operator} if $S(f)\ge 0$ whenever $f\ge 0$, $S({\bf 1}_X)={\bf 1}_Y$ and $S^*({\bf 1}_Y)={\bf 1}_X$.

\begin{proposition}\cite[Theorems 13.2 and 13.8]{eisner2015operator}\label{pr:Basic Markov}
The class of Markov operators is closed under adjoints, composition and pointwise limits, in the sense that if $S_n:L^2(X,\mu)\to L^2(Y,\nu)$ are Markov operators for every $n\in \mathbb{N}$ and there is $S:L^2(X,\mu)\to L^2(Y,\nu)$ such that
$$||S_n(f)-S(f)||_2\to 0$$
for every $f\in L^2(X,\mu)$, then $S$ is a Markov operator.
Moreover, every Markov operator is a contraction, i.e., its norm is bounded by $1$.
\end{proposition}

We say that $P:L^2(X,\mu)\to L^2(X,\mu)$ is a {\it Markov projection} if it is an orthogonal projection and a Markov operator (see \cite[Section~13.3]{eisner2015operator}).

\begin{theorem}[Structure of Markov projections] \cite[Theorem~13.20]{eisner2015operator}\label{th:MarkovProj}
Let $(X,\mathcal{B})$ be a standard Borel space and $\mu$ be a Borel probability measure.
There is a one-to-one correspondence between
\begin{enumerate}
	\item Markov projections,
	\item $\Theta_\mu$, the relatively complete sub-$\sigma$-algebras of $\mathcal{B}$.
\end{enumerate}
The correspondence is given as
$$P\mapsto \{A\in\mathcal{B}:P({\bf 1}_A)={\bf 1}_A\} \ \text{ and } \ \mathcal{C}\mapsto \mathbb{E}({-}|\mathcal{C}).$$
\end{theorem}

\begin{theorem}[Mean Ergodic Theorem] \cite[Theorem~8.6, Example~13.24]{eisner2015operator}\label{th:Mean Ergodic}
Let $(X,\mathcal{B})$ be a standard Borel space, $\mu$ be a Borel probability measure and $S:L^2(X,\mu)\to L^2(X,\mu)$ be a Markov operator.
Then
$$ \left\lVert \frac{1}{n}\sum_{k\in [n]}S^k(f)-P(f)\right\lVert_2\to 0$$
for every $f\in L^2(X,\mu)$, where $P$ is the orthogonal projection onto the closed subspace $\{g\in L^2(X,\mu):S(g)=g\}$.
\end{theorem}

\section{Quotient Spaces}\label{App E}

\begin{theorem}\label{th:Quotients and Markov}
Let $(X,\mathcal{B})$ be a standard Borel space, $\mu$ be a Borel probability measure on $X$ and $\mathcal{C}\in \Theta_\mu$.
There is a standard Borel space $(X/\mathcal{C},\mathcal{C}')$, a Borel probability measure $\mu/\mathcal{C}$ on $X/\mathcal{C}$, measurable surjection $q_\mathcal{C}:X\to X/\mathcal{C}$, and Markov operators
$$S_\mathcal{C}:L^2(X,\mu)\to L^2(X/\mathcal{C},\mu/\mathcal{C}) \ \text{ and } \  I_{\mathcal{C}}:L^2(X/\mathcal{C},\mu/\mathcal{C})\to L^2(X,\mu)$$
such that
\begin{enumerate}
	\item $\mu/\mathcal{C}$ is the push-forward of $\mu$ via $q_\mathcal{C}$,
	\item $S_\mathcal{C}^*=I_{\mathcal{C}}$,
	\item $S_\mathcal{C}\circ \mathbb{E}({-}|\mathcal{C})=S_\mathcal{C}$,
	\item $I_\mathcal{C}$ is an isometry onto $L^2(X,\mathcal{C},\mu)$,
	\item $I_\mathcal{C}\circ S_\mathcal{C}=\mathbb{E}({-}|\mathcal{C})$,
	\item $S_\mathcal{C}\circ I_\mathcal{C}$ is the identity on $L^2(X/\mathcal{C},\mu/\mathcal{C})$,
	\item $I_\mathcal{C}(f)(x)=f(q_\mathcal{C}(x))$ for every $f\in L^2(X/\mathcal{C},\mu/\mathcal{C})$.
\end{enumerate}
\end{theorem}
\begin{proof}
The existence of $(X/\mathcal{C},\mathcal{C}')$, $\mu/\mathcal{C}$ and $q_\mathcal{C}$ follows from \cite[Exercise~17.43~ii)]{kechris1995classical}.
Define $I_\mathcal{C}$ by the condition (7).
Then it is easy to see that $I_\mathcal{C}$ is a Markov embedding by \cite[Section~12.2, Theorem~13.9]{eisner2015operator} and all the other properties follow from \cite[Section~13.2 and~13.3]{eisner2015operator}.
\end{proof}

The next results imply that the space $X/\mathcal{C}$ is unique up to a ``$\mu$-negligible part''.

\begin{corollary}\label{cor:Quotient}
Let $(X,\mathcal{B})$ and $(Y,\mathcal{D})$ be standard Borel spaces.
Suppose that $\mu$ is a Borel probability measure on $X$ and $f:X\to Y$ is a Borel function.
Write $\mathcal{C}\in \Theta_\mu$ for the minimum relatively complete sub-$\sigma$-algebra that makes $f$ measurable.
Then for every $g_0\in L^2(X,\mathcal{C},\mu)$ there is a Borel map $g_1:Y\to \mathbb{C}$ such that $g_0(x)=\left(g_1\circ f\right)(x)$ for $\mu$-almost every $x\in X$.
\end{corollary}
\begin{proof}
Put $\nu=f_*\mu\in \mathscr{P}(Y)$ and note that by \cite[Theorem~21.10]{kechris1995classical} there is a $Y_0\in \mathcal{D}$ such that $Y_0\subseteq f(X)$ and $\nu(Y_0)=1$.
Then use Theorem~\ref{th:Quotients and Markov}.
\end{proof}

We say that a map $S:L^2(X,\mu)\to L^2(Y,\nu)$ is a {\it Markov isomorphism} if it is a Markov operator that is an isometrical bijection (see \cite[Section~12.2]{eisner2015operator}).

\begin{theorem}\label{th:Markov Isomorphism}
Let $(X,\mathcal{B})$, $(Y,\mathcal{D})$ be a standard Borel spaces, $\mu$ be a Borel probability measure on $X$ and $\nu$ be a Borel probability measure on $Y$.
Then there is a one-to-one correspondence between
\begin{enumerate}
	\item Markov isomorphisms $S:L^2(X,\mu)\to L^2(Y,\nu)$,
	\item measure preserving almost bijections $i: X\to Y$.
\end{enumerate}
The correspondence from (2) to (1) is given as
$$i\mapsto S_i(f)(x)=f(i^{-1}(x)). $$
\end{theorem}
\begin{proof}
It follows from \cite[Theorem~12.10]{eisner2015operator} that there is a correspondence between Markov isomorphisms and measure algebra isomorphisms.
It is a standard fact (see \cite[Theorem~1.9]{Kerr}) that every measure algebra isomorphism is induced by a measurable measure preserving almost bijection under the assumption that the spaces are standard Borel.
\end{proof}

\end{document}